\documentclass[12pt]{article}
\usepackage{amsmath,mathtools,float}
\usepackage{appendix}
\usepackage{multirow}
\usepackage{url}
\usepackage{graphicx}
\usepackage[all]{xy}
\usepackage{amsfonts}
\date{}
\usepackage{amsthm}
\usepackage{amssymb}
\usepackage{fullpage}
\theoremstyle{plain}

\newtheorem{theorem}{Theorem}[section]
\newtheorem{definition}{Definition}[section]

\newtheorem{remark}{Remark}[section]

\numberwithin{equation}{section}

\def\R{\hbox{{\rm I}\kern-0.2em{\rm R}\kern0.2em}}

\def\bn{\begin{equation}}
\def\en{\end{equation}}
\def\bny{\begin{eqnarray}}
\def\eny{\end{eqnarray}}
\def\be{\begin{eqnarray*}}
\def\ee{\end{eqnarray*}}
\def\bc{\begin{center}}
\def\ec{\end{center}}

\def\({\left(}
\def\){\right  )}
\def\[{\left[}
\def\]{\right]}
\def\bc{\begin{center}}
\def\ec{\end{center}}

\def\bn{\begin{equation}}
\def\en{\end{equation}}
\def\bny{\begin{eqnarray}}
\def\eny{\end{eqnarray}}
\def\be{\begin{eqnarray*}}
\def\ee{\end{eqnarray*}}
\def\bdn{\begin{dfn}}
\def\edn{\end{dfn}}
\def\btm{\begin{thm}}
\def\etm{\end{thm}}
\def\bpf{\begin{proof}}
\def\epf{\end{proof}}
\def\bpn{\begin{pro}}
\def\epn{\end{pro}}
\def\brk{\begin{rem}}
\def\erk{\end{rem}}
\def\bcy{\begin{cor}}
\def\ecy{\end{cor}}
\def\blm{\begin{lem}}\def\elm{\end{lem}}
\def\bex{\begin{exm}}
\def\eex{\end{exm}}

 \def\R{{\hat R}}
\begin{document}

\bc {\bf A simplification and generalization of Elsayed's two-dimensional system of third order difference equations.
  }\ec
\medskip
\bc
M. Folly-Gbetoula and D. Nyirenda 
 \vspace{1cm}
\\School of Mathematics, University of the Witwatersrand, Johannesburg, South Africa.\\

\ec
\begin{abstract}
\noindent  A full Lie analysis of a  system of third-order difference equations
is performed. Explicit solutions, expressed in terms of the initial values, are derived. Furthermore, we give sufficient conditions for existence of 2-periodic and 4-periodic solutions in certain cases. Our results generalize and simplify the work by Elsayed and Ibrahim [E. M. Elsayed and T.F. Ibrahim, Periodicity and solutions for some systems of nonlinear rational difference equations, \textit{Hacettepe Journal of Mathematics and Statistics} \textbf{44:6} (2015), 1361-1390].
\end{abstract}
\section{Introduction} \setcounter{equation}{0}
The area of difference equations has attracted many researchers recently. Methods for solving difference equations have been developed ( see \cite{2,3,4, hydon0, Mensah, Mensah2}) and Lie symmetry approach is one of them. One of the most useful algorithms for computing symmetries of difference equations is due to Hydon (see \cite{hydon0}). The Lie symmetry group of a system of difference equations is the largest group of point transformations acting on the space of dependent and independent variables that leave the equations unchanged. Thus an element of such a group, maps a solution of the difference equation onto another solution.  In this method, the order of the difference equation is reduced and using the invariance of the equation under group transformations or via the similarity variables, one can find exact solutions.\par \noindent
In this paper, by applying Lie symmetry method, we generalize some results in \cite{EI}, where Elsayed and Ibrahim investigated the periodic nature and the form of the solutions of nonlinear difference equations systems of order three:
\begin{align}\label{1.0}
x_{n+1}=\frac{ x_{n-2}y_{n-1}}{y_n(\pm 1\pm 1 x_{n-2}y_{n-1})}, y_{n+1}=\frac{ y_{n-2}x_{n-1}}{x_n(\pm 1\pm 1 y_{n-2}x_{n-1})}.
\end{align}
We study the system
\begin{align}\label{1.1}
x_{n+1}=\frac{x_{n-2}y_{n-1}}{y_{n}(a_n+b_nx_{n-2}y_{n-1})}, \quad y_{n+1}=\frac{y_{n-2}x_{n-1}}{x_{n}(c_n+d_ny_{n-2}x_{n-1})},
\end{align}  where $(a_n)_{n\in \mathbb{N}_0},\;(b_n)_{n\in \mathbb{N}_0}$ are non-zero real sequences, $x_{-2},\,x_{-1},\,x_0, y_{-2},\,y_{-1}$ and $y_0$ are initial values. Because of the definitions and notation we want to use, we study the equivalent system
\begin{align}\label{1.2}
x_{n+3}=\frac{x_ny_{n+1}}{y_{n+2}(A_n+B_nx_ny_{n+1})}, \quad y_{n+3}=\frac{y_nx_{n+1}}{x_{n+2}(C_n+D_ny_nx_{n+1})},
\end{align}  where $(A_n)_{n\in \mathbb{N}_0},\;(B_n)_{n\in \mathbb{N}_0}$ are non-zero real sequences.\par \noindent
In coming up with the solutions of (\ref{1.1}) using Lie symmetry method, we first find the Lie algebra of \eqref{1.1}.  We then reduce the order of the equations by utilizing the invariants, and later use iterations to deduce the solutions.
\subsection{Preliminaries}
In this section, we give a background to Lie symmetry analysis. The notation used comes from \cite{hydon0}.
\begin{definition}\cite{P.Olver} Let $G$ be a local group of transformations acting on a manifold $M$. A subset $\mathcal{S}\subset M$ is called G-invariant, and $G$ is called symmetry group of $\mathcal{S}$, if whenever $x\in \mathcal{S} $, and $g\in G$ is such that $g\cdot x$ is defined, then $g\cdot x \in \mathcal{S}$.
\end{definition}
\begin{definition}\cite{P.Olver} Let $G$ be a connected group of transformations acting on a manifold $M$. A smooth real-valued function $\zeta: M\rightarrow \mathbb{R}$ is an invariant function for $G$ if and only if $$X(\zeta)=0\qquad \text { for all } \qquad  x\in M,$$
and every infinitesimal generator $X$ of $G$.
\end{definition}
\begin{definition}\cite{hydon0}
A parameterized set of point transformations,
\begin{equation}
\Gamma_{\varepsilon} :x\mapsto \hat{x}(x;\varepsilon),
\label{eq: b}
\end{equation}
where $x=x_i, $ $i=1,\dots,p$ are continuous variables, is a one-parameter local Lie group of transformations if the following conditions are satisfied:
\begin{enumerate}
\item $\Gamma_0$ is the identity map if $\hat{x}=x$ when $\varepsilon=0$
\item $\Gamma_a\Gamma_b=\Gamma_{a+b}$ for every $a$ and $b$ sufficiently close to 0
\item Each $\hat{x_i}$ can be represented as a Taylor series (in a neighborhood of $\varepsilon=0$ that is determined by $x$), and therefore
\end{enumerate}
\begin{equation}
\hat{x_i}(x:\varepsilon)=x_i+\varepsilon \xi _i(x)+O(\varepsilon ^2), i=1,...,p.
\label{eq: c}
\end{equation}
\end{definition}
Consider the system of ordinary difference equations
\begin{align}\label{general}
\begin{cases}
x_{n+3}=&\Omega _1(x_n, x_{n+1}, x_{n+2}, y_n, y_{n+1}, y_{n+2}), \\
y_{n+3}=&\Omega _2(x_n, x_{n+1}, x_{n+2}, y_n, y_{n+1}, y_{n+2}), \quad n\in D
\end{cases}
\end{align}
for some smooth function $\Omega = (\Omega_1, \Omega_2)$ and a regular domain $D\subset \mathbb{Z}$.
To find a symmetry group of \eqref{general}, we consider the group of point transformations given by
\begin{equation}\label{Gtransfo}
G_{\varepsilon}: (x_n,y_n) \mapsto(x_n+\varepsilon Q_1 (n,x_n),y_n+\varepsilon Q_2 (n,y_n)),
\end{equation}
where $\varepsilon$ is the parameter and $Q_i,\; i=1,2$, the continuous functions which we shall refer to as characteristics. Let
\begin{align}\label{Ngener}
\mathcal{X}= & Q_1(n,x_n)\frac{\partial}{ \partial x_n}+ Q_2(n,y_n)\frac{\partial}{ \partial y_n}
\end{align}
be the corresponding infinitesimal of $G _{\varepsilon}$ with the $k$-th extension
\begin{equation}
X=
Q_1\frac{\partial}{ \partial x_n}+ Q_2\frac{\partial}{ \partial y_n}+ SQ_1\frac{\partial}{ \partial x_{n+1}} +SQ_2\frac{\partial}{ \partial y_{n+1}} + S^2Q_1\frac{\partial}{ \partial x_{n+2}} +S^2Q_2\frac{\partial}{ \partial y_{n+2}}.
\end{equation}
Note that $S$ is the forward shift operator, acting on $n$ as follows: $S:n \rightarrow n+1$. Further, the linearized symmetry conditions are given by
\begin{align}\label{LSC}
\mathcal{S}^{(3)} Q_1- X \Omega _1=0,\quad 
\mathcal{S}^{(3)} Q_2- X \Omega _2=0.
\end{align}
 Once the characteristics $Q_i$ are known, the invariant $\zeta  _i$ may be obtained  by introducing the canonical coordinate \cite{JV}
 \begin{align}\label{cano}
s_n= \int{\frac{dx_n}{Q_1(n,x_n)}} \quad \text{ and }\quad t_n= \int{\frac{dy_n}{Q_2(n,y_n)}} .
 \end{align}
In general, the constraints on the constants in the characteristics give one a clear idea (without any lucky guess) about the perfect choice of invariants.
\section{Symmetries, reductions and exact solutions of \eqref{1.2}}
Consider the system \eqref{1.2}, that is,
\begin{align}\label{1.2'}
x_{n+3}=\frac{x_ny_{n+1}}{y_{n+2}(A_n+B_nx_ny_{n+1})}, \quad 
 y_{n+3}=\frac{y_nx_{n+1}}{x_{n+2}(C_n+D_ny_nx_{n+1})}.
\end{align}
\subsection{Symmetries}
To get the symmetries, we impose the infinitesimal criterion of invariance \eqref{LSC} to get
\begin{subequations}\label{a1A1}
\begin{align}\label{a1}
& (S^{3}Q_1 )+ \frac{(B_n {x_n}y_{n+1}+A_n){x_n}y_{n+1}(S^2Q_2)-A_n {x_n}y_{n+2}(SQ_2)-A_ny_{n+1}y_{n+2}Q_1}{(B_n {x_n} {y_{n+1}} + A_n)^2}=0,
\end{align}
\begin{align}\label{A1}
&(S^{3}Q_2 )+ \frac{(D_n {y_n}x_{n+1}+C_n){y_n}x_{n+1}(S^2Q_1)-C_n {y_n}x_{n+2}(SQ_1)-C_nx_{n+1}x_{n+2}Q_2}{(D_n {y_n} {x_{n+1}} + C_n)^2}=0.
\end{align}
\end{subequations}
These are functional equations for the characteristics $Q_i,\;i=1,2$. To eliminate the first undesirable arguments $x_{n+3}$ and $y_{n+3}$, we act the differential operator
$\frac{\partial}{\partial x_n}-\frac{y_{n+1}}{x_n}\frac{\partial\quad}{\partial y_{n+1}}$
on \eqref{a1} and
$\frac{\partial}{\partial y_n}-\frac{x_{n+1}}{y_n}\frac{\partial\quad}{\partial x_{n+1}}$
on  \eqref{A1}, the following expressions are obtained after simplification:
\begin{subequations}\label{a3A3}
\begin{align}\label{a3}
y_{n+1}(SQ_2)'-y_{n+1}Q_1'-SQ_2+\frac{y_{n+1}}{x_n}Q_1=0
\end{align}
and
\begin{align}\label{A3}
x_{n+1}(SQ_1)'-x_{n+1}Q_2'-SQ_1+\frac{x_{n+1}}{y_n}Q_2=0.
\end{align}
\end{subequations}
To eliminate the second undesirable arguments $x_{n+1}$ and $y_{n+1}$, we differentiate \eqref{a3} with respect to $x_n$ and differentiate \eqref{A3} with respect to $y_n$. Solving the resulting differential equations for $Q_1$ and $Q_2$ gives
\begin{subequations}\label{a5A5}
\begin{align}\label{a5}
Q_1(n,x_n)=\alpha_n x_n + \beta _n x_n \ln x_n
\end{align}
and
\begin{align}\label{A5}
Q_2(n,y_{n})=\lambda _n y_n + \mu _n y_n\ln y_n,
\end{align}
\end{subequations}
where $\alpha _n,\; \beta _n,\; \lambda _n $ and $\mu _n$ are arbitrary functions of $n$. We gain more information on these functions by substituting equations in \eqref{a5A5} in equations in \eqref{a1A1}. The resulting equations can solved by the method of separation which yields the following systems:
\begin{subequations}\label{a8A8}
\begin{align}\label{a8}
\begin{split}
y_{n+1}{x_n}^2&: \lambda _{n+2} +\alpha _{n+3}=0\\
x_n&: \lambda _{n+2} +\alpha _{n+3}-\alpha _n -\lambda _{n+1} =0
\end{split}
\end{align}
and
\begin{align}\label{A8}
\begin{split}
x_{n+1}{y_n}^2&: \beta _{n+2} +\lambda _{n+3}=0\\
y_n&: \beta _{n+2} +\lambda _{n+3}-\lambda _n -\alpha _{n+1} =0
\end{split}
\end{align}
\end{subequations}
or simply
\begin{align}\label{a9A9}
\lambda _n + \alpha _{n+1} =0,\quad 
\text { and } \quad 
\alpha _n +\lambda _{n+1} =0.
\end{align}
It turned out that $\beta _n$ and $\mu _n$ are zero.
\noindent
From \eqref{a8A8}, we can see that $\lambda_{n+2}-\lambda _n =0$. Thus the solutions of \eqref{a9A9}  are given by
$
\alpha_ n = -\lambda _n =(-1)^n
$ and therefore the characteristics are as follows
\begin{align}\label{a10}
Q_1= & (-1)^n x_n,\quad  Q_2=-(-1)^ny_n.
\end{align}
The Lie algebra of \eqref{1.1} is then spanned by
\begin{align}\label{a10}
\mathcal{X}= & (-1)^n x_n\frac{\partial}{ \partial x_n}-(-1)^n y_n\frac{\partial}{ \partial y_n}.
\end{align}
\subsection{Reduction and solutions}
Using \eqref{cano} and \eqref{a10}, we found that the canonical coordinates are given by
\begin{align}\label{11}
s_n =(-1)^n\ln |x_n | \quad \text{and} \quad t_n=(-1)^{n+1}\ln| y_n|.
\end{align}
We replace $\alpha _n$ and its shift (resp $\lambda _n$ and its shift) with $s _n \alpha _n $ and its shift (resp $t_n \lambda _n $ and its shift) in \eqref{a9A9} and the left hand sides of the resulting equations give the invariants:
\begin{align}\label{12}
\tilde{U}_n=\ln|x_ny_{n+1}| \quad \text{and} \quad \tilde{V}_n=\ln|x_{n+1}y_n|.
\end{align}
The reader can verify that $X[\tilde{U}_n]=X[\tilde{V}_n]=0$. For the sake of convenience, we consider
\begin{align}\label{13}
U_n =\exp\{-\tilde{U}_n\}\quad \text{and} \quad
V_n=\exp\{-\tilde{V}_n\}
\end{align}
instead or simply $U_n=\pm1/(x_ny_{n+1})$ and $V_n=\pm1/(x_{n+1}y_n)$. Using the plus sign, this leads to
\begin{align}\label{14}
V_{n + 2} = A_{n}U_{n} + B_{n} \,\,\text{and}\,\, U_{n+2} = C_{n}V_{n} + D_{n}.
\end{align}
For equations in \eqref{14} , replace $V_{n}$ in the second equation by $A_{n-2}U_{n - 2} + B_{n - 2}$ to get
$U_{n+2} = C_{n}A_{n - 2}U_{n - 2} + C_{n}B_{n - 2} + D_{n}$ which implies
\begin{align}\label{15}
 U_{n+4} = C_{n + 2}A_{n}U_{n} + C_{n + 2}B_{n} + D_{n + 2}.
\end{align}
Iterating several times, one obtains
\begin{align}\label{16}
U_{4n + j} = U_{j}\prod\limits_{\mathclap{k_1 = 0}}^{\mathclap{n - 1}} A_{4k_1 + j}C_{4k_1 + j + 2} + \sum\limits_{l = 0}^{n - 1}\left((B_{4l + j}C_{4l + j + 2} + D_{4l + j + 2})\prod\limits_{\mathclap{k_2 = l + 1}}^{\mathclap{n - 1}}A_{4k_2 + j}C_{4k_2 + j + 2} \right)
\end{align}
where $j = 0,1,2,3$.
Similarly, we have
\begin{align}\label{17}
V_{4n + j} & = V_{j}\prod\limits_{\mathclap{k_1 = 0}}^{\mathclap{n - 1}} A_{4k_1 + j + 2}C_{4k_1 + j} + \sum\limits_{l = 0}^{n - 1}\left((A_{4l + j + 2}D_{4l + j} + B_{4l + j + 2})\prod\limits_{\mathclap{k_2 = l + 1}}^{\mathclap{n - 1}}A_{4k_2 + j + 2}C_{4k_2 + j} \right)
\end{align}
where  $0\leq j \leq 3$.\\
\noindent The equations  $U_n= 1/(x_ny_{n+1})$ and $V_n= 1/(x_{n+1}y_n)$ imply that
\begin{align}\label{18}
x_{n + 2} = \frac{U_{n}}{V_{n+1}}x_{n}\quad \text{and}\quad y_{n + 2} = \frac{V_{n}}{U_{n+1}}y_{n}
\end{align}
which yield
\begin{align}\label{19}
x_{2n + j} = x_{j}\prod_{i = 0}^{n - 1}\frac{ U_{2i + j}}{V_{2i + j + 1}}\quad \text{and}\quad y_{2n + j} = y_{j}\prod_{i = 0}^{n - 1}\frac{V_{2i + j}}{U_{2i + j + 1}}
\end{align}
where $j = 0, 1$.\\
Hence, we have
\begin{align}\label{20a}
x_{4n + j} & = x_{j}\prod_{i = 0}^{n - 1}\frac{ U_{4i + j}U_{4i + 2 + j}}{V_{4i + j + 1}V_{4i + j + 3}}
\end{align}
and
\begin{align}\label{20b}
y_{4n + j} = y_{j}\prod_{i = 0}^{n - 1}\frac{V_{4i + j} V_{4i + 2 + j} }{U_{4i + j + 1}U_{4i + j + 3}}
\end{align}
where $j = 0,1,2,3$.\\
\noindent Substituting specific values of $j$ and using \eqref{20a} and \eqref{20b}, we have;
\begingroup\makeatletter\def\f@size{9}\check@mathfonts
\def\maketag@@@#1{\hbox{\m@th\large\normalfont#1}}%
\begin{align}\label{20a0}
x_{4n} & = x_{0}\prod_{s = 0}^{n - 1}\frac{ U_{4s}U_{4s + 2}}{V_{4s + 1}V_{4s + 3}} \nonumber\\
       & = x_{0}\prod_{s = 0}^{n - 1}\frac{ U_{0}\prod\limits_{k_1 = 0}^{s - 1} A_{4k_1}C_{4k_1 + 2} + \sum\limits_{l = 0}^{s - 1}\left((B_{4l}C_{4l + 2} + D_{4l  + 2})\prod\limits_{k_2 = l + 1}^{s - 1}A_{4k_2 }C_{4k_2 + 2} \right) }{ V_{1}\prod\limits_{k_1 = 0}^{s - 1} A_{4k_1 + 3}C_{4k_1 + 1} + \sum\limits_{l = 0}^{s - 1}\left((A_{4l + 3}D_{4l + 1} + B_{4l + 3})\prod\limits_{k_2 = l + 1}^{s - 1}A_{4k_2 + 3}C_{4k_2 + 1} \right)}\nonumber \\
       & \quad \times \frac{U_{2}\prod\limits_{k_1 = 0}^{s - 1} A_{4k_1 + 2}C_{4k_1 + 4} + \sum\limits_{l = 0}^{s - 1}\left((B_{4l + 2}C_{4l + 4} + D_{4l + 4})\prod\limits_{k_2 = l + 1}^{s - 1}A_{4k_2 + 2}C_{4k_2 + 4} \right)}{V_{3}\prod\limits_{k_1 = 0}^{s - 1} A_{4k_1 + 5}C_{4k_1 + 3} + \sum\limits_{l = 0}^{s - 1}\left((A_{4l + 5}D_{4l + 3} + B_{4l + 5})\prod\limits_{k_2 = l + 1}^{s - 1}A_{4k_2 + 5}C_{4k_2 + 3} \right)},
\end{align}
\begin{align}\label{20b0}
y_{4n} & = y_{0}\prod_{s = 0}^{n - 1}\frac{ V_{0}\prod\limits_{k_1 = 0}^{s - 1} A_{4k_1 + 2}C_{4k_1} + \sum\limits_{l = 0}^{s - 1}\left((A_{4l + 2}D_{4l} + B_{4l + 2})\prod\limits_{k_2 = l + 1}^{s - 1}A_{4k_2 + 2}C_{4k_2} \right)}{ U_{1}\prod\limits_{k_1 = 0}^{s - 1} A_{4k_1 + 1}C_{4k_1 + 3} + \sum\limits_{l = 0}^{s - 1}\left((B_{4l + 1}C_{4l + 3} + D_{4l + 3})\prod\limits_{k_2 = l + 1}^{s - 1}A_{4k_2 + 1}C_{4k_2 + 3} \right)}\nonumber\\
& \quad \times \frac{V_{2}\prod\limits_{k_1 = 0}^{s - 1} A_{4k_1 + 4}C_{4k_1 + 2} + \sum\limits_{l = 0}^{s- 1}\left((A_{4l + 4}D_{4l + 2} + B_{4l + 4})\prod\limits_{k_2 = l + 1}^{s - 1}A_{4k_2 + 4}C_{4k_2 + 2} \right)}{U_{3}\prod\limits_{k_1 = 0}^{s - 1} A_{4k_1 + 3}C_{4k_1 + 5} + \sum\limits_{l = 0}^{s - 1}\left((B_{4l + 3}C_{4l + 5} + D_{4l + 5})\prod\limits_{k_2 = l + 1}^{s - 1}A_{4k_2 + 3}C_{4k_2 + 5} \right)}.
\end{align}
\begin{align}\label{20a1}
x_{4n + 1} & = x_{1}\prod_{s = 0}^{n - 1}\frac{ U_{4s + 1}U_{4s + 3}}{V_{4s + 2}V_{4s + 4}}\nonumber \\
           & = x_{1}\prod_{s = 0}^{n - 1}\frac{ U_{1}\prod\limits_{k_1 = 0}^{s - 1} A_{4k_1 + 1}C_{4k_1 + 3} + \sum\limits_{l = 0}^{s - 1}\left((B_{4l + 1}C_{4l + 3} + D_{4l + 3})\prod\limits_{k_2 = l + 1}^{s - 1}A_{4k_2 + 1}C_{4k_2 + 3} \right)}{V_{2}\prod\limits_{k_1 = 0}^{s - 1} A_{4k_1 + 4}C_{4k_1 + 2} + \sum\limits_{l = 0}^{s- 1}\left((A_{4l + 4}D_{4l + 2} + B_{4l + 4})\prod\limits_{k_2 = l + 1}^{s - 1}A_{4k_2 + 4}C_{4k_2 + 2} \right)} \nonumber\\
           & \quad \times \frac{ U_{3}\prod\limits_{k_1 = 0}^{s - 1} A_{4k_1 + 3}C_{4k_1 + 5} + \sum\limits_{l = 0}^{s - 1}\left((B_{4l + 3}C_{4l + 5} + D_{4l + 5})\prod\limits_{k_2 = l + 1}^{s - 1}A_{4k_2 + 3}C_{4k_2 + 5} \right)}{V_{0}\prod\limits_{k_1 = 0}^{s} A_{4k_1 + 2}C_{4k_1} + \sum\limits_{l = 0}^{s}\left((A_{4l + 2}D_{4l} + B_{4l + 2})\prod\limits_{k_2 = l + 1}^{s}A_{4k_2 + 2}C_{4k_2} \right)    },
\end{align}
\begin{align}\label{20b1}
y_{4n + 1}  & = y_{1}\prod_{i = 0}^{n - 1}\frac{V_{4s + 1} V_{4s + 3} }{U_{4s + 2}U_{4s + 4}} \nonumber\\
            & = y_{1}\prod_{i = 0}^{n - 1}\frac{ V_{1}\prod\limits_{k_1 = 0}^{s - 1} A_{4k_1 + 3}C_{4k_1 + 1} + \sum\limits_{l = 0}^{s - 1}\left((A_{4l + 3}D_{4l + 1} + B_{4l + 3})\prod\limits_{k_2 = l + 1}^{s - 1}A_{4k_2 + 3}C_{4k_2 + 1} \right) }{ U_{2}\prod\limits_{k_1 = 0}^{s - 1} A_{4k_1 + 2}C_{4k_1 + 4} + \sum\limits_{l = 0}^{s - 1}\left((B_{4l + 2}C_{4l + 4} + D_{4l + 4})\prod\limits_{k_2 = l + 1}^{s - 1}A_{4k_2 + 2}C_{4k_2 + 4} \right)} \nonumber \\
            & \quad \times \frac{V_{3}\prod\limits_{k_1 = 0}^{s - 1} A_{4k_1 + 5}C_{4k_1 + 3} + \sum\limits_{l = 0}^{s - 1}\left((A_{4l + 5}D_{4l + 3} + B_{4l + 5})\prod\limits_{k_2 = l + 1}^{s - 1}A_{4k_2 + 5}C_{4k_2 + 3} \right)}{U_{0}\prod\limits_{k_1 = 0}^{s} A_{4k_1}C_{4k_1 + 2} + \sum\limits_{l = 0}^{s}\left((B_{4l}C_{4l + 2} + D_{4l  + 2})\prod\limits_{k_2 = l + 1}^{s}A_{4k_2 }C_{4k_2 + 2} \right)},
\end{align}
\begin{align}\label{20a2}
x_{4n + 2} & = x_{2}\prod_{s = 0}^{n - 1}\frac{ U_{4s + 2}U_{4s + 4}}{V_{4s + 3}V_{4s + 5}} \nonumber\\
           & = x_{2}\prod_{s = 0}^{n - 1}\frac{U_{2}\prod\limits_{k_1 = 0}^{s - 1} A_{4k_1 + 2}C_{4k_1 + 4} + \sum\limits_{l = 0}^{s - 1}\left((B_{4l + 2}C_{4l + 4} + D_{4l + 4})\prod\limits_{k_2 = l + 1}^{s - 1}A_{4k_2 + 2}C_{4k_2 + 4} \right)}{V_{3}\prod\limits_{k_1 = 0}^{s - 1} A_{4k_1 + 5}C_{4k_1 + 3} + \sum\limits_{l = 0}^{s - 1}\left((A_{4l + 5}D_{4l + 3} + B_{4l + 5})\prod\limits_{k_2 = l + 1}^{s - 1}A_{4k_2 + 5}C_{4k_2 + 3} \right)} \nonumber\\
           & \quad \times \frac{  U_{0}\prod\limits_{k_1 = 0}^{s} A_{4k_1}C_{4k_1 + 2} + \sum\limits_{l = 0}^{s}\left((B_{4l}C_{4l + 2} + D_{4l  + 2})\prod\limits_{k_2 = l + 1}^{s}A_{4k_2 }C_{4k_2 + 2} \right) }{V_{1}\prod\limits_{k_1 = 0}^{s} A_{4k_1 + 3}C_{4k_1 + 1} + \sum\limits_{l = 0}^{s}\left((A_{4l + 3}D_{4l + 1} + B_{4l + 3})\prod\limits_{k_2 = l + 1}^{s}A_{4k_2 + 3}C_{4k_2 + 1} \right)},
\end{align}
\begin{align}\label{20b2}
y_{4n + 2} & = y_{2}\prod_{s = 0}^{n - 1}\frac{V_{4s + 2} V_{4s + 4} }{U_{4s + 3}U_{4s + 5}} \nonumber\\
           & = y_{2}\prod_{s = 0}^{n - 1}\frac{ V_{2}\prod\limits_{k_1 = 0}^{s - 1} A_{4k_1 + 4}C_{4k_1 + 2} + \sum\limits_{l = 0}^{s- 1}\left((A_{4l + 4}D_{4l + 2} + B_{4l + 4})\prod\limits_{k_2 = l + 1}^{s - 1}A_{4k_2 + 4}C_{4k_2 + 2} \right) }{U_{3}\prod\limits_{k_1 = 0}^{s - 1} A_{4k_1 + 3}C_{4k_1 + 5} + \sum\limits_{l = 0}^{s - 1}\left((B_{4l + 3}C_{4l + 5} + D_{4l + 5})\prod\limits_{k_2 = l + 1}^{s - 1}A_{4k_2 + 3}C_{4k_2 + 5} \right)}  \nonumber\end{align}\begin{align}
           & \quad \times \frac{ V_{0}\prod\limits_{k_1 = 0}^{s} A_{4k_1 + 2}C_{4k_1} + \sum\limits_{l = 0}^{s}\left((A_{4l + 2}D_{4l} + B_{4l + 2})\prod\limits_{k_2 = l + 1}^{s}A_{4k_2 + 2}C_{4k_2} \right)}{U_{1}\prod\limits_{k_1 = 0}^{s} A_{4k_1 + 1}C_{4k_1 + 3} + \sum\limits_{l = 0}^{s}\left((B_{4l + 1}C_{4l + 3} + D_{4l + 3})\prod\limits_{k_2 = l + 1}^{s}A_{4k_2 + 1}C_{4k_2 + 3} \right)},
\end{align}
\begin{align}\label{20a3}
x_{4n + 3} & = x_{3}\prod_{s = 0}^{n - 1}\frac{ U_{4s + 3}U_{4s + 5}}{V_{4s + 4}V_{4s + 6}} \nonumber\\
           & = x_{3}\prod_{s = 0}^{n - 1}\frac{ U_{3}\prod\limits_{k_1 = 0}^{s - 1} A_{4k_1 + 3}C_{4k_1 + 5} + \sum\limits_{l = 0}^{s - 1}\left((B_{4l + 3}C_{4l + 5} + D_{4l + 5})\prod\limits_{k_2 = l + 1}^{s - 1}A_{4k_2 + 3}C_{4k_2 + 5} \right)}{V_{0}\prod\limits_{k_1 = 0}^{s} A_{4k_1 + 2}C_{4k_1} + \sum\limits_{l = 0}^{s}\left((A_{4l + 2}D_{4l} + B_{4l + 2})\prod\limits_{k_2 = l + 1}^{s}A_{4k_2 + 2}C_{4k_2} \right)} \nonumber\\
           & \quad \times  \frac{ U_{1}\prod\limits_{k_1 = 0}^{s} A_{4k_1 + 1}C_{4k_1 + 3} + \sum\limits_{l = 0}^{s}\left((B_{4l + 1}C_{4l + 3} + D_{4l + 3})\prod\limits_{k_2 = l + 1}^{s}A_{4k_2 + 1}C_{4k_2 + 3} \right)}{ V_{2}\prod\limits_{k_1 = 0}^{s} A_{4k_1 + 4}C_{4k_1 + 2} + \sum\limits_{l = 0}^{s}\left((A_{4l + 4}D_{4l + 2} + B_{4l + 4})\prod\limits_{k_2 = l + 1}^{s}A_{4k_2 + 4}C_{4k_2 + 2} \right)},
\end{align}
\begin{align}\label{20b3}
y_{4n + 3} & = y_{3}\prod_{s = 0}^{n - 1}\frac{V_{4s + 3} V_{4s + 5} }{U_{4s + 4}U_{4s + 6}}\nonumber\\
           & = y_{3}\prod_{s = 0}^{n - 1}\frac{V_{3}\prod\limits_{k_1 = 0}^{s - 1} A_{4k_1 + 5}C_{4k_1 + 3} + \sum\limits_{l = 0}^{s - 1}\left((A_{4l + 5}D_{4l + 3} + B_{4l + 5})\prod\limits_{k_2 = l + 1}^{s - 1}A_{4k_2 + 5}C_{4k_2 + 3} \right)}{U_{0}\prod\limits_{k_1 = 0}^{s} A_{4k_1}C_{4k_1 + 2} + \sum\limits_{l = 0}^{s}\left((B_{4l}C_{4l + 2} + D_{4l  + 2})\prod\limits_{k_2 = l + 1}^{s}A_{4k_2 }C_{4k_2 + 2} \right)}  \nonumber\\
           & \quad \times \frac{V_{1}\prod\limits_{k_1 = 0}^{s} A_{4k_1 + 3}C_{4k_1 + 1} + \sum\limits_{l = 0}^{s}\left((A_{4l + 3}D_{4l + 1} + B_{4l + 3})\prod\limits_{k_2 = l + 1}^{s}A_{4k_2 + 3}C_{4k_2 + 1} \right)}{U_{2}\prod\limits_{k_1 = 0}^{s} A_{4k_1 + 2}C_{4k_1 + 4} + \sum\limits_{l = 0}^{s}\left((B_{4l + 2}C_{4l + 4} + D_{4l + 4})\prod\limits_{k_2 = l + 1}^{s}A_{4k_2 + 2}C_{4k_2 + 4} \right)}.
\end{align}
\endgroup
\section{Solutions of equation \eqref{1.1}}
From the previous section, replacing $x_{n}$ with $x_{n -2}$, $y_{n}$ with $y_{n - 2}$, $V_{n}$ with $\frac{1}{x_{n - 1}y_{n-2}}$, $U_{n}$ with $\frac{1}{x_{n-2}y_{n - 1}}$ and $A_{n}, B_{n}, C_{n}, D_{n}$ with $a_{n}, b_{n}, c_{n} d_{n}$, respectively, we obtain the solutions for \eqref{1.1} as follows:
\begingroup\makeatletter\def\f@size{9}\check@mathfonts
\def\maketag@@@#1{\hbox{\m@th\large\normalfont#1}}%
\begin{align}
x_{4n - 2} & = \frac{x_{2}^{n}}{x_{-2}^{n - 1}}\prod_{s = 0}^{n - 1}\frac{\prod\limits_{k_1 = 0}^{s - 1} a_{4k_1}c_{4k_1 + 2} + x_{-2}y_{-1}\sum\limits_{l = 0}^{s - 1}\left((b_{4l}c_{4l + 2} + d_{4l  + 2})\prod\limits_{k_2 = l + 1}^{s - 1}a_{4k_2 }c_{4k_2 + 2} \right) }{\prod\limits_{k_1 = 0}^{s - 1} a_{4k_1 + 3}c_{4k_1 + 1} + x_0y_{-1}\sum\limits_{l = 0}^{s - 1}\left((a_{4l + 3}d_{4l + 1} + b_{4l + 3})\prod\limits_{k_2 = l + 1}^{s - 1}a_{4k_2 + 3}c_{4k_2 + 1} \right)} \nonumber\\
       & \quad \times \frac{\prod\limits_{k_1 = 0}^{s - 1} a_{4k_1 + 2}c_{4k_1 + 4} + x_0y_1\sum\limits_{l = 0}^{s - 1}\left((b_{4l + 2}c_{4l + 4} + d_{4l + 4})\prod\limits_{k_2 = l + 1}^{s - 1}a_{4k_2 + 2}c_{4k_2 + 4} \right)}{\prod\limits_{k_1 = 0}^{s - 1} a_{4k_1 + 5}c_{4k_1 + 3} + x_2y_1\sum\limits_{l = 0}^{s - 1}\left((a_{4l + 5}d_{4l + 3} + b_{4l + 5})\prod\limits_{k_2 = l + 1}^{s - 1}a_{4k_2 + 5}c_{4k_2 + 3} \right)} \nonumber\end{align}
    \begin{align}\label{21a0p}
       & = \frac{x_{2}^{n}}{x_{-2}^{n - 1}}\prod_{s = 0}^{n - 1}\frac{\prod\limits_{k_1 = 0}^{s - 1} a_{4k_1}c_{4k_1 + 2} + x_{-2}y_{-1}\sum\limits_{l = 0}^{s - 1}\left((b_{4l}c_{4l + 2} + d_{4l  + 2})\prod\limits_{k_2 = l + 1}^{s - 1}a_{4k_2 }c_{4k_2 + 2} \right) }{\prod\limits_{k_1 = 0}^{s - 1} a_{4k_1 + 3}c_{4k_1 + 1} + x_0y_{-1}\sum\limits_{l = 0}^{s - 1}(a_{4l + 3}d_{4l + 1} + b_{4l + 3})\prod\limits_{k_2 = l + 1}^{s - 1}a_{4k_2 + 3}c_{4k_2 + 1}} \nonumber\\
       & \quad \times \frac{    \prod\limits_{k_1 = 0}^{s - 1} a_{4k_1 + 2}c_{4k_1 + 4} + \frac{x_{-1}y_{-2}}{c_0 + d_{0}x_{-1}y_{-2}}\sum\limits_{l = 0}^{s - 1}\left((b_{4l + 2}c_{4l + 4} + d_{4l + 4})\prod\limits_{k_2 = l + 1}^{s - 1}a_{4k_2 + 2}c_{4k_2 + 4} \right)  }{  \prod\limits_{k_1 = 0}^{s - 1} a_{4k_1 + 5}c_{4k_1 + 3} + \frac{x_{-1}y_0}{a_1 + b_{1}x_{-1}y_0}\sum\limits_{l = 0}^{s - 1}\left((a_{4l + 5}d_{4l + 3} + b_{4l + 5})\prod\limits_{k_2 = l + 1}^{s - 1}a_{4k_2 + 5}c_{4k_2 + 3} \right)} \nonumber\\
       & = \frac{1}{x_{-2}^{n - 1}}\prod_{s = 0}^{n - 1}\Bigg(\frac{\prod\limits_{k_1 = 0}^{s - 1} a_{4k_1}c_{4k_1 + 2} + x_{-2}y_{-1}\sum\limits_{l = 0}^{s - 1}\left((b_{4l}c_{4l + 2} + d_{4l  + 2})\prod\limits_{k_2 = l + 1}^{s - 1}a_{4k_2 }c_{4k_2 + 2} \right) }{\prod\limits_{k_1 = 0}^{s - 1} a_{4k_1 + 3}c_{4k_1 + 1} + x_0y_{-1}\sum\limits_{l = 0}^{s - 1}\left((a_{4l + 3}d_{4l + 1} + b_{4l + 3})\prod\limits_{k_2 = l + 1}^{s - 1}a_{4k_2 + 3}c_{4k_2 + 1} \right)} \times\nonumber\\
       & %
         \frac{    c_{0}\prod\limits_{\mathclap{k_1 = 0}}^{\mathclap{s - 1}} a_{4k_1 + 2}c_{4k_1 + 4}  + \left(d_{0}\prod\limits_{\mathclap{k_1 = 0}}^{\mathclap{s - 1}} a_{4k_1 + 2}c_{4k_1 + 4} + \sum\limits_{l = 0}^{s - 1}(b_{4l + 2}c_{4l + 4} + d_{4l + 4})\prod\limits_{\mathclap{k_2 = l + 1}}^{\mathclap{s - 1}}a_{4k_2 + 2}c_{4k_2 + 4} \right)x_{-1}y_{-2}  }{  a_{1}\prod\limits_{\mathclap{k_1 = 0}}^{\mathclap{s - 1}} a_{4k_1 + 5}c_{4k_1 + 3} + \left( b_{1}\prod\limits_{\mathclap{k_1 = 0}}^{s - 1} a_{4k_1 + 5}c_{4k_1 + 3}  + \sum\limits_{l = 0}^{s - 1}(a_{4l + 5}d_{4l + 3} + b_{4l + 5})\prod\limits_{\mathclap{k_2 = l + 1}}^{\mathclap{s - 1}}a_{4k_2 + 5}c_{4k_2 + 3} \right)x_{-1}y_0} \Bigg) \nonumber\\
       & \times \left(\frac{a_1 + b_{1}x_{-1}y_{0}}{c_0 + d_{0}x_{-1}y_{-2}}\right)^{n}  \left(\frac{x_0y_0(c_0 + d_0x_{-1}y_{-2})}{y_{-2}(a_1 + b_1x_{-1}y_0)}\right)^n\nonumber \\
       & = \frac{x_0^{n}y_0^{n}}{x_{-2}^{n - 1}y_{-2}^{n}}\prod_{s = 0}^{n - 1}\frac{\prod\limits_{k_1 = 0}^{s - 1} a_{4k_1}c_{4k_1 + 2} + x_{-2}y_{-1}\sum\limits_{l = 0}^{s - 1}\left((b_{4l}c_{4l + 2} + d_{4l  + 2})\prod\limits_{k_2 = l + 1}^{s - 1}a_{4k_2 }c_{4k_2 + 2} \right) }{\prod\limits_{k_1 = 0}^{s - 1} a_{4k_1 + 3}c_{4k_1 + 1} + x_0y_{-1}\sum\limits_{l = 0}^{s - 1}\left((a_{4l + 3}d_{4l + 1} + b_{4l + 3})\prod\limits_{k_2 = l + 1}^{s - 1}a_{4k_2 + 3}c_{4k_2 + 1} \right)}\times \nonumber\\
       &  \frac{    c_{0}\prod\limits_{\mathclap{k_1 = 0}}^{\mathclap{s - 1}} a_{4k_1 + 2}c_{4k_1 + 4}  + \left(d_{0}\prod\limits_{\mathclap{k_1 = 0}}^{\mathclap{s - 1}} a_{4k_1 + 2}c_{4k_1 + 4} + \sum\limits_{l = 0}^{s - 1}(b_{4l + 2}c_{4l + 4} + d_{4l + 4})\prod\limits_{\mathclap{k_2 = l + 1}}^{\mathclap{s - 1}}a_{4k_2 + 2}c_{4k_2 + 4} \right)x_{-1}y_{-2}  }{  a_{1}\prod\limits_{\mathclap{k_1 = 0}}^{\mathclap{s - 1}} a_{4k_1 + 5}c_{4k_1 + 3} + \left( b_{1}\prod\limits_{\mathclap{k_1 = 0}}^{\mathclap{s - 1}} a_{4k_1 + 5}c_{4k_1 + 3}  + \sum\limits_{l = 0}^{s - 1}(a_{4l + 5}d_{4l + 3} + b_{4l + 5})\prod\limits_{\mathclap{k_2 = l + 1}}^{\mathclap{s - 1}}a_{4k_2 + 5}c_{4k_2 + 3} \right)x_{-1}y_0},
\end{align}
\begin{align}
y_{4n - 2} & = \frac{y_2^{n}}{y_{-2}^{n - 1}}\prod_{s = 0}^{n - 1}\frac{\prod\limits_{k_1 = 0}^{s - 1} a_{4k_1 + 2}c_{4k_1} + x_{-1}y_{-2}\sum\limits_{l = 0}^{s - 1}\left((a_{4l + 2}d_{4l} + b_{4l + 2})\prod\limits_{k_2 = l + 1}^{s - 1}a_{4k_2 + 2}c_{4k_2} \right)}{ \prod\limits_{k_1 = 0}^{s - 1} a_{4k_1 + 1}c_{4k_1 + 3} + x_{-1}y_0\sum\limits_{l = 0}^{s - 1}\left((b_{4l + 1}c_{4l + 3} + d_{4l + 3})\prod\limits_{k_2 = l + 1}^{s - 1}a_{4k_2 + 1}c_{4k_2 + 3} \right)}\nonumber\\
& \quad \times \frac{\prod\limits_{k_1 = 0}^{s - 1} a_{4k_1 + 4}c_{4k_1 + 2} + x_1y_0\sum\limits_{l = 0}^{s- 1}\left((a_{4l + 4}d_{4l + 2} + b_{4l + 4})\prod\limits_{k_2 = l + 1}^{s - 1}a_{4k_2 + 4}c_{4k_2 + 2} \right)}{\prod\limits_{k_1 = 0}^{s - 1} a_{4k_1 + 3}c_{4k_1 + 5} + x_1y_2\sum\limits_{l = 0}^{s - 1}\left((b_{4l + 3}c_{4l + 5} + d_{4l + 5})\prod\limits_{k_2 = l + 1}^{s - 1}a_{4k_2 + 3}c_{4k_2 + 5} \right)}\nonumber\\
         & = \frac{y_2^{n}}{y_{-2}^{n - 1}}\prod_{s = 0}^{n - 1}\frac{\prod\limits_{k_1 = 0}^{s - 1} a_{4k_1 + 2}c_{4k_1} + x_{-1}y_{-2}\sum\limits_{l = 0}^{s - 1}\left((a_{4l + 2}d_{4l} + b_{4l + 2})\prod\limits_{k_2 = l + 1}^{s - 1}a_{4k_2 + 2}c_{4k_2} \right)}{ \prod\limits_{k_1 = 0}^{s - 1} a_{4k_1 + 1}c_{4k_1 + 3} + x_{-1}y_0\sum\limits_{l = 0}^{s - 1}\left((b_{4l + 1}c_{4l + 3} + d_{4l + 3})\prod\limits_{k_2 = l + 1}^{s - 1}a_{4k_2 + 1}c_{4k_2 + 3} \right)}\nonumber\end{align}\begin{align}
& \quad \times \frac{   \prod\limits_{k_1 = 0}^{s - 1} a_{4k_1 + 4}c_{4k_1 + 2} + \frac{x_{-2}y_{-1}}{a_0 + b_0x_{-2}y_{-1}}\sum\limits_{l = 0}^{s- 1}\left((a_{4l + 4}d_{4l + 2} + b_{4l + 4})\prod\limits_{k_2 = l + 1}^{s - 1}a_{4k_2 + 4}c_{4k_2 + 2} \right)}{\prod\limits_{k_1 = 0}^{s - 1} a_{4k_1 + 3}c_{4k_1 + 5} + \frac{x_0y_{-1}}{c_1 + d_1x_0y_{-1}}\sum\limits_{l = 0}^{s - 1}\left((b_{4l + 3}c_{4l + 5} + d_{4l + 5})\prod\limits_{k_2 = l + 1}^{s - 1}a_{4k_2 + 3}c_{4k_2 + 5} \right)}\nonumber\\
       & = \frac{x_0^{n}y_0^{n}}{x_{-2}^{n}y_{-2}^{n - 1}}\prod_{s = 0}^{n - 1}\frac{\prod\limits_{k_1 = 0}^{s - 1} a_{4k_1 + 2}c_{4k_1} + x_{-1}y_{-2}\sum\limits_{l = 0}^{s - 1}\left((a_{4l + 2}d_{4l} + b_{4l + 2})\prod\limits_{k_2 = l + 1}^{s - 1}a_{4k_2 + 2}c_{4k_2} \right)}{ \prod\limits_{k_1 = 0}^{s - 1} a_{4k_1 + 1}c_{4k_1 + 3} + x_{-1}y_0\sum\limits_{l = 0}^{s - 1}\left((b_{4l + 1}c_{4l + 3} + d_{4l + 3})\prod\limits_{k_2 = l + 1}^{s - 1}a_{4k_2 + 1}c_{4k_2 + 3} \right)}\times\nonumber\\
&  \frac{   a_{0}\prod\limits_{\mathclap{k_1 = 0}}^{\mathclap{s - 1}} a_{4k_1 + 4}c_{4k_1 + 2} + \left(b_{0}\prod\limits_{\mathclap{k_1 = 0}}^{\mathclap{s - 1}} a_{4k_1 + 4}c_{4k_1 + 2} + \sum\limits_{l = 0}^{s- 1}(a_{4l + 4}d_{4l + 2} + b_{4l + 4})\prod\limits_{\mathclap{k_2 = l + 1}}^{\mathclap{s - 1}}a_{4k_2 + 4}c_{4k_2 + 2} \right)x_{-2}y_{-1}}{ c_{1}\prod\limits_{\mathclap{k_1 = 0}}^{\mathclap{s - 1}} a_{4k_1 + 3}c_{4k_1 + 5} +  \left( d_{1}\prod\limits_{\mathclap{k_1 = 0}}^{\mathclap{s - 1}} a_{4k_1 + 3}c_{4k_1 + 5}  + \sum\limits_{l = 0}^{s - 1}(b_{4l + 3}c_{4l + 5} + d_{4l + 5})\prod\limits_{\mathclap{k_2 = l + 1}}^{\mathclap{s - 1}}a_{4k_2 + 3}c_{4k_2 + 5} \right)x_{0}y_{-1}},\label{21b0p}
\end{align}
\begin{align}
x_{4n - 1} & = \frac{x_{-1}y_{-2}^{n}}{y_2^{n}}\prod_{s = 0}^{n - 1}\frac{\prod\limits_{k_1 = 0}^{s - 1} a_{4k_1 + 1}c_{4k_1 + 3} + x_{-1}y_0\sum\limits_{l = 0}^{s - 1}\left((b_{4l + 1}c_{4l + 3} + d_{4l + 3})\prod\limits_{k_2 = l + 1}^{s - 1}a_{4k_2 + 1}c_{4k_2 + 3} \right)}{  \prod\limits_{k_1 = 0}^{s - 1} a_{4k_1 + 4}c_{4k_1 + 2} + x_1y_0\sum\limits_{l = 0}^{s- 1}\left((a_{4l + 4}d_{4l + 2} + b_{4l + 4})\prod\limits_{k_2 = l + 1}^{s - 1}a_{4k_2 + 4}c_{4k_2 + 2} \right)}\nonumber \\
           & \quad \times \frac{\prod\limits_{k_1 = 0}^{s - 1} a_{4k_1 + 3}c_{4k_1 + 5} + x_1y_2\sum\limits_{l = 0}^{s - 1}\left((b_{4l + 3}c_{4l + 5} + d_{4l + 5})\prod\limits_{k_2 = l + 1}^{s - 1}a_{4k_2 + 3}c_{4k_2 + 5} \right)}{\prod\limits_{k_1 = 0}^{s} a_{4k_1 + 2}c_{4k_1} + x_{-1}y_{-2}\sum\limits_{l = 0}^{s}\left((a_{4l + 2}d_{4l} + b_{4l + 2})\prod\limits_{k_2 = l + 1}^{s}a_{4k_2 + 2}c_{4k_2} \right)}\nonumber \\ 
           & = \frac{x_{-1}y_{-2}^{n}}{y_2^{n}}\prod_{s = 0}^{n - 1}\frac{\prod\limits_{k_1 = 0}^{s - 1} a_{4k_1 + 1}c_{4k_1 + 3} + x_{-1}y_0\sum\limits_{l = 0}^{s - 1}\left((b_{4l + 1}c_{4l + 3} + d_{4l + 3})\prod\limits_{k_2 = l + 1}^{s - 1}a_{4k_2 + 1}c_{4k_2 + 3} \right)}{  \prod\limits_{k_1 = 0}^{s - 1} a_{4k_1 + 4}c_{4k_1 + 2} + \frac{x_{-2}y_{-1}}{a_0 + b_0x_{-2}y_{-1}}\sum\limits_{l = 0}^{s- 1}\left((a_{4l + 4}d_{4l + 2} + b_{4l + 4})\prod\limits_{k_2 = l + 1}^{s - 1}a_{4k_2 + 4}c_{4k_2 + 2} \right)}\nonumber \\
           & \quad \times \frac{\prod\limits_{k_1 = 0}^{s - 1} a_{4k_1 + 3}c_{4k_1 + 5} + \frac{x_0y_{-1}}{c_1 + d_1x_0y_{-1}}\sum\limits_{l = 0}^{s - 1}\left((b_{4l + 3}c_{4l + 5} + d_{4l + 5})\prod\limits_{k_2 = l + 1}^{s - 1}a_{4k_2 + 3}c_{4k_2 + 5} \right)}{\prod\limits_{k_1 = 0}^{s} a_{4k_1 + 2}c_{4k_1} + x_{-1}y_{-2}\sum\limits_{l = 0}^{s}\left((a_{4l + 2}d_{4l} + b_{4l + 2})\prod\limits_{k_2 = l + 1}^{s}a_{4k_2 + 2}c_{4k_2} \right)}\nonumber \\
          & = \prod_{\mathclap{s = 0}}^{\mathclap{n - 1}}\left(\frac{\prod\limits_{\mathclap{k_1 = 0}}^{\mathclap{s - 1}} a_{4k_1 + 1}c_{4k_1 + 3} + x_{-1}y_0\sum\limits_{l = 0}^{s - 1}\left((b_{4l + 1}c_{4l + 3} + d_{4l + 3})\prod\limits_{\mathclap{k_2 = l + 1}}^{\mathclap{s - 1}}a_{4k_2 + 1}c_{4k_2 + 3} \right)}{ a_{0}\prod\limits_{\mathclap{k_1 = 0}}^{\mathclap{s - 1}} a_{4k_1 + 4}c_{4k_1 + 2} + \left(b_{0}\prod\limits_{\mathclap{k_1 = 0}}^{\mathclap{s - 1}} a_{4k_1 + 4}c_{4k_1 + 2}    +  \sum\limits_{l = 0}^{s- 1}(a_{4l + 4}d_{4l + 2} + b_{4l + 4})\prod\limits_{\mathclap{k_2 = l + 1}}^{\mathclap{s - 1}}a_{4k_2 + 4}c_{4k_2 + 2} \right)x_{-2}y_{-1}}\right.\nonumber \\
           & \left.\times \frac{   c_1\prod\limits_{\mathclap{k_1 = 0}}^{\mathclap{s - 1}} a_{4k_1 + 3}c_{4k_1 + 5} +  \left(d_1\prod\limits_{\mathclap{k_1 = 0}}^{\mathclap{s - 1}} a_{4k_1 + 3}c_{4k_1 + 5}  + \sum\limits_{l = 0}^{s - 1}(b_{4l + 3}c_{4l + 5} + d_{4l + 5})\prod\limits_{\mathclap{k_2 = l + 1}}^{\mathclap{s - 1}}a_{4k_2 + 3}c_{4k_2 + 5} \right)x_0y_{-1}}{\prod\limits_{\mathclap{k_1 = 0}}^{s} a_{4k_1 + 2}c_{4k_1} + x_{-1}y_{-2}\sum\limits_{l = 0}^{s}\left((a_{4l + 2}d_{4l} + b_{4l + 2})\prod\limits_{k_2 = l + 1}^{s}a_{4k_2 + 2}c_{4k_2} \right)}\right)\nonumber\\
           & \quad \quad \quad \quad \times \frac{x_{-1}x_{-2}^{n}y_{-2}^n}{x_0^{n}y_0^{n}},\label{21a1p}
\end{align}
\begin{align}
y_{4n - 1}  & = \frac{y_{-1}x_{-2}^{n}}{x_2^{n}}\prod_{i = 0}^{n - 1}\frac{\prod\limits_{k_1 = 0}^{s - 1} a_{4k_1 + 3}c_{4k_1 + 1} + x_0y_{-1}\sum\limits_{l = 0}^{s - 1}\left((a_{4l + 3}d_{4l + 1} + b_{4l + 3})\prod\limits_{k_2 = l + 1}^{s - 1}a_{4k_2 + 3}c_{4k_2 + 1} \right) }{\prod\limits_{k_1 = 0}^{s - 1} a_{4k_1 + 2}c_{4k_1 + 4} + x_0y_1\sum\limits_{l = 0}^{s - 1}\left((b_{4l + 2}c_{4l + 4} + d_{4l + 4})\prod\limits_{k_2 = l + 1}^{s - 1}a_{4k_2 + 2}c_{4k_2 + 4} \right)} \nonumber \\
            & \quad \times \frac{\prod\limits_{k_1 = 0}^{s - 1} a_{4k_1 + 5}c_{4k_1 + 3} + x_2y_1\sum\limits_{l = 0}^{s - 1}\left((a_{4l + 5}d_{4l + 3} + b_{4l + 5})\prod\limits_{k_2 = l + 1}^{s - 1}a_{4k_2 + 5}c_{4k_2 + 3} \right)}{\prod\limits_{k_1 = 0}^{s} a_{4k_1}c_{4k_1 + 2} + x_{-2}y_{-1}\sum\limits_{l = 0}^{s}\left((b_{4l}c_{4l + 2} + d_{4l  + 2})\prod\limits_{k_2 = l + 1}^{s}a_{4k_2 }c_{4k_2 + 2} \right)}\nonumber\\
           & = \prod_{i = 0}^{n - 1}\left(\frac{\prod\limits_{\mathclap{k_1 = 0}}^{\mathclap{s - 1}} a_{4k_1 + 3}c_{4k_1 + 1} + x_0y_{-1}\sum\limits_{l = 0}^{s - 1}\left((a_{4l + 3}d_{4l + 1} + b_{4l + 3})\prod\limits_{\mathclap{k_2 = l + 1}}^{\mathclap{s - 1}}a_{4k_2 + 3}c_{4k_2 + 1} \right) }{   c_0\prod\limits_{\mathclap{k_1 = 0}}^{\mathclap{s - 1}} a_{4k_1 + 2}c_{4k_1 + 4} + \left( d_0\prod\limits_{\mathclap{k_1 = 0}}^{\mathclap{s - 1}} a_{4k_1 + 2}c_{4k_1 + 4}    +  \sum\limits_{l = 0}^{s - 1}(b_{4l + 2}c_{4l + 4} + d_{4l + 4})\prod\limits_{\mathclap{k_2 = l + 1}}^{\mathclap{s - 1}}a_{4k_2 + 2}c_{4k_2 + 4} \right)x_{-1}y_{-2}   } \right.\nonumber \\
            & \left. \times \frac{  a_1\prod\limits_{\mathclap{k_1 = 0}}^{\mathclap{s - 1}} a_{4k_1 + 5}c_{4k_1 + 3} + \left(b_1\prod\limits_{\mathclap{k_1 = 0}}^{\mathclap{s - 1}} a_{4k_1 + 5}c_{4k_1 + 3}   + \sum\limits_{l = 0}^{s - 1}(a_{4l + 5}d_{4l + 3} + b_{4l + 5})\prod\limits_{\mathclap{k_2 = l + 1}}^{\mathclap{s - 1}}a_{4k_2 + 5}c_{4k_2 + 3} \right)x_{-1}y_0}{\prod\limits_{\mathclap{k_1 = 0}}^{s} a_{4k_1}c_{4k_1 + 2} + x_{-2}y_{-1}\sum\limits_{l = 0}^{s}\left((b_{4l}c_{4l + 2} + d_{4l  + 2})\prod\limits_{\mathclap{k_2 = l + 1}}^{s}a_{4k_2 }c_{4k_2 + 2} \right)}\right)\nonumber\\
            & \quad \quad \quad \times \frac{y_{-1}x_{-2}^{n}y_{-2}^n}{x_0^{n}y_0^{n}},\label{21b1p}
\end{align}
\begin{align}
x_{4n} & = \frac{x_{0}x_2^{n}}{x_{-2}^{n}}\prod_{s = 0}^{n - 1}\frac{\prod\limits_{k_1 = 0}^{s - 1} a_{4k_1 + 2}c_{4k_1 + 4} + x_0y_1\sum\limits_{l = 0}^{s - 1}\left((b_{4l + 2}c_{4l + 4} + d_{4l + 4})\prod\limits_{k_2 = l + 1}^{s - 1}a_{4k_2 + 2}c_{4k_2 + 4} \right)}{\prod\limits_{k_1 = 0}^{s - 1} a_{4k_1 + 5}c_{4k_1 + 3} + x_2y_1\sum\limits_{l = 0}^{s - 1}\left((a_{4l + 5}d_{4l + 3} + b_{4l + 5})\prod\limits_{k_2 = l + 1}^{s - 1}a_{4k_2 + 5}c_{4k_2 + 3} \right)} \nonumber\\
          & \quad \times \frac{\prod\limits_{k_1 = 0}^{s} a_{4k_1}c_{4k_1 + 2} + x_{-2}y_{-1}\sum\limits_{l = 0}^{s}\left((b_{4l}c_{4l + 2} + d_{4l  + 2})\prod\limits_{k_2 = l + 1}^{s}a_{4k_2 }c_{4k_2 + 2} \right) }{\prod\limits_{k_1 = 0}^{s} a_{4k_1 + 3}c_{4k_1 + 1} + x_0y_{-1}\sum\limits_{l = 0}^{s}\left((a_{4l + 3}d_{4l + 1} + b_{4l + 3})\prod\limits_{k_2 = l + 1}^{s}a_{4k_2 + 3}c_{4k_2 + 1} \right)}\nonumber\\
           & = \prod_{s = 0}^{n - 1}\left(\frac{   c_0\prod\limits_{\mathclap{k_1 = 0}}^{\mathclap{s - 1}} a_{4k_1 + 2}c_{4k_1 + 4} + \left(d_0\prod\limits_{\mathclap{k_1 = 0}}^{\mathclap{s - 1}} a_{4k_1 + 2}c_{4k_1 + 4}     +  \sum\limits_{l = 0}^{s - 1}(b_{4l + 2}c_{4l + 4} + d_{4l + 4})\prod\limits_{\mathclap{k_2 = l + 1}}^{\mathclap{s - 1}}a_{4k_2 + 2}c_{4k_2 + 4} \right)x_{-1}y_{-2}}{ a_1\prod\limits_{\mathclap{k_1 = 0}}^{\mathclap{s - 1}} a_{4k_1 + 5}c_{4k_1 + 3} + \left(b_1\prod\limits_{\mathclap{k_1 = 0}}^{s - 1} a_{4k_1 + 5}c_{4k_1 + 3}    +  \sum\limits_{l = 0}^{s - 1}(a_{4l + 5}d_{4l + 3} + b_{4l + 5})\prod\limits_{\mathclap{k_2 = l + 1}}^{\mathclap{s - 1}}a_{4k_2 + 5}c_{4k_2 + 3} \right)x_{-1}y_0}\right. \nonumber\\
           &\left. \times \frac{\prod\limits_{k_1 = 0}^{s} a_{4k_1}c_{4k_1 + 2} + x_{-2}y_{-1}\sum\limits_{l = 0}^{s}\left((b_{4l}c_{4l + 2} + d_{4l  + 2})\prod\limits_{k_2 = l + 1}^{s}a_{4k_2 }c_{4k_2 + 2} \right) }{\prod\limits_{k_1 = 0}^{s} a_{4k_1 + 3}c_{4k_1 + 1} + x_0y_{-1}\sum\limits_{l = 0}^{s}\left((a_{4l + 3}d_{4l + 1} + b_{4l + 3})\prod\limits_{k_2 = l + 1}^{s}a_{4k_2 + 3}c_{4k_2 + 1} \right)}\right)       \frac{x_{0}^{n + 1}y_0^{n}}{x_{-2}^{n}y_{-2}^{n}},\label{21a2p}
\end{align}
\begin{align}
y_{4n} & = \frac{y_{0}y_2^{n}}{y_{-2}^{n}}\prod_{\mathclap{s = 0}}^{\mathclap{n - 1}}\frac{\prod\limits_{\mathclap{k_1 = 0}}^{\mathclap{s - 1}} a_{4k_1 + 4}c_{4k_1 + 2} + x_1y_0\sum\limits_{l = 0}^{s- 1}\left((a_{4l + 4}d_{4l + 2} + b_{4l + 4})\prod\limits_{k_2 = l + 1}^{s - 1}a_{4k_2 + 4}c_{4k_2 + 2} \right) }{\prod\limits_{k_1 = 0}^{s - 1} a_{4k_1 + 3}c_{4k_1 + 5} + x_1y_2\sum\limits_{l = 0}^{s - 1}\left((b_{4l + 3}c_{4l + 5} + d_{4l + 5})\prod\limits_{k_2 = l + 1}^{s - 1}a_{4k_2 + 3}c_{4k_2 + 5} \right)} \nonumber \\
           & \quad \times \frac{\prod\limits_{\mathclap{k_1 = 0}}^{s} a_{4k_1 + 2}c_{4k_1} + x_{-1}y_{-2}\sum\limits_{l = 0}^{s}\left((a_{4l + 2}d_{4l} + b_{4l + 2})\prod\limits_{\mathclap{k_2 = l + 1}}^{s}a_{4k_2 + 2}c_{4k_2} \right)}{\prod\limits_{\mathclap{k_1 = 0}}^{s} a_{4k_1 + 1}c_{4k_1 + 3} + x_{-1}y_0\sum\limits_{l = 0}^{s}\left((b_{4l + 1}c_{4l + 3} + d_{4l + 3})\prod\limits_{\mathclap{k_2 = l + 1}}^{s}a_{4k_2 + 1}c_{4k_2 + 3} \right)},\nonumber\\
           & = \prod_{\mathclap{s = 0}}^{\mathclap{n - 1}}\left(\frac{ a_0\prod\limits_{\mathclap{k_1 = 0}}^{\mathclap{s - 1}} a_{4k_1 + 4}c_{4k_1 + 2} + \left(b_0\prod\limits_{\mathclap{k_1 = 0}}^{\mathclap{s - 1}} a_{4k_1 + 4}c_{4k_1 + 2}   +  \sum\limits_{l = 0}^{s- 1}(a_{4l + 4}d_{4l + 2} + b_{4l + 4})\prod\limits_{\mathclap{k_2 = l + 1}}^{\mathclap{s - 1}}a_{4k_2 + 4}c_{4k_2 + 2} \right)x_{-2}y_{-1} }{ c_1\prod\limits_{\mathclap{k_1 = 0}}^{\mathclap{s - 1}} a_{4k_1 + 3}c_{4k_1 + 5} + \left(d_1\prod\limits_{\mathclap{k_1 = 0}}^{\mathclap{s - 1}} a_{4k_1 + 3}c_{4k_1 + 5}   + \sum\limits_{l = 0}^{s - 1}(b_{4l + 3}c_{4l + 5} + d_{4l + 5})\prod\limits_{\mathclap{k_2 = l + 1}}^{\mathclap{s - 1}}a_{4k_2 + 3}c_{4k_2 + 5} \right)x_0y_{-1}} \right.\nonumber \\
           & \left. \quad \times \frac{\prod\limits_{k_1 = 0}^{s} a_{4k_1 + 2}c_{4k_1} + x_{-1}y_{-2}\sum\limits_{l = 0}^{s}\left((a_{4l + 2}d_{4l} + b_{4l + 2})\prod\limits_{k_2 = l + 1}^{s}a_{4k_2 + 2}c_{4k_2} \right)}{\prod\limits_{k_1 = 0}^{s} a_{4k_1 + 1}c_{4k_1 + 3} + x_{-1}y_0\sum\limits_{l = 0}^{s}\left((b_{4l + 1}c_{4l + 3} + d_{4l + 3})\prod\limits_{k_2 = l + 1}^{s}a_{4k_2 + 1}c_{4k_2 + 3} \right)}\right)         \frac{x_0^{n}y_{0}^{n + 1}}{x_{-2}^{n}y_{-2}^{n}},\label{21b2p}
\end{align}
\begin{align}
x_{4n + 1} & = \frac{x_{1}y_{-2}^{n}}{y_2^{n}}\prod_{s = 0}^{n - 1}\frac{\prod\limits_{k_1 = 0}^{s - 1} a_{4k_1 + 3}c_{4k_1 + 5} + x_1y_2\sum\limits_{l = 0}^{s - 1}\left((b_{4l + 3}c_{4l + 5} + d_{4l + 5})\prod\limits_{k_2 = l + 1}^{s - 1}a_{4k_2 + 3}c_{4k_2 + 5} \right)}{\prod\limits_{k_1 = 0}^{s} a_{4k_1 + 2}c_{4k_1} + x_{-1}y_{-2}\sum\limits_{l = 0}^{s}\left((a_{4l + 2}d_{4l} + b_{4l + 2})\prod\limits_{k_2 = l + 1}^{s}a_{4k_2 + 2}c_{4k_2} \right)} \nonumber\\
           & \quad \times  \frac{\prod\limits_{k_1 = 0}^{s} a_{4k_1 + 1}c_{4k_1 + 3} + x_{-1}y_0\sum\limits_{l = 0}^{s}\left((b_{4l + 1}c_{4l + 3} + d_{4l + 3})\prod\limits_{k_2 = l + 1}^{s}a_{4k_2 + 1}c_{4k_2 + 3} \right)}{ \prod\limits_{k_1 = 0}^{s} a_{4k_1 + 4}c_{4k_1 + 2} + x_1y_0\sum\limits_{l = 0}^{s}\left((a_{4l + 4}d_{4l + 2} + b_{4l + 4})\prod\limits_{k_2 = l + 1}^{s}a_{4k_2 + 4}c_{4k_2 + 2} \right)},\nonumber\\
           & = \frac{x_{1}y_{-2}^{n}}{y_2^{n}}\prod_{s = 0}^{n - 1}\frac{  \prod\limits_{k_1 = 0}^{s - 1} a_{4k_1 + 3}c_{4k_1 + 5} + \frac{x_0y_{-1}}{c_1 + d_1x_0y_{-1}}\sum\limits_{l = 0}^{s - 1}\left((b_{4l + 3}c_{4l + 5} + d_{4l + 5})\prod\limits_{k_2 = l + 1}^{s - 1}a_{4k_2 + 3}c_{4k_2 + 5} \right)}{\prod\limits_{k_1 = 0}^{s} a_{4k_1 + 2}c_{4k_1} + x_{-1}y_{-2}\sum\limits_{l = 0}^{s}\left((a_{4l + 2}d_{4l} + b_{4l + 2})\prod\limits_{k_2 = l + 1}^{s}a_{4k_2 + 2}c_{4k_2} \right)} \nonumber\\
           & \quad \times  \frac{\prod\limits_{k_1 = 0}^{s} a_{4k_1 + 1}c_{4k_1 + 3} + x_{-1}y_0\sum\limits_{l = 0}^{s}\left((b_{4l + 1}c_{4l + 3} + d_{4l + 3})\prod\limits_{k_2 = l + 1}^{s}a_{4k_2 + 1}c_{4k_2 + 3} \right)}{   \prod\limits_{k_1 = 0}^{s} a_{4k_1 + 4}c_{4k_1 + 2} +  \frac{x_{-2}y_{-1}}{a_0 + b_0x_{-2}y_{-1}}\sum\limits_{l = 0}^{s}\left((a_{4l + 4}d_{4l + 2} + b_{4l + 4})\prod\limits_{k_2 = l + 1}^{s}a_{4k_2 + 4}c_{4k_2 + 2} \right)},\nonumber \\
           & =  \prod_{s = 0}^{n - 1}\left(\frac{  c_1\prod\limits_{\mathclap{k_1 = 0}}^{\mathclap{s - 1}} a_{4k_1 + 3}c_{4k_1 + 5} + \left(d_1\prod\limits_{\mathclap{k_1 = 0}}^{\mathclap{s - 1}} a_{4k_1 + 3}c_{4k_1 + 5}    + \sum\limits_{l = 0}^{s - 1}(b_{4l + 3}c_{4l + 5} + d_{4l + 5})\prod\limits_{\mathclap{k_2 = l + 1}}^{\mathclap{s - 1}}a_{4k_2 + 3}c_{4k_2 + 5} \right)x_0y_{-1}}{\prod\limits_{\mathclap{k_1 = 0}}^{s} a_{4k_1 + 2}c_{4k_1} + x_{-1}y_{-2}\sum\limits_{l = 0}^{s}\left((a_{4l + 2}d_{4l} + b_{4l + 2})\prod\limits_{\mathclap{k_2 = l + 1}}^{s}a_{4k_2 + 2}c_{4k_2} \right)}\right. \nonumber\end{align}\begin{align}
           & \left. \times  \frac{\prod\limits_{\mathclap{k_1 = 0}}^{s} a_{4k_1 + 1}c_{4k_1 + 3} + x_{-1}y_0\sum\limits_{l = 0}^{s}\left((b_{4l + 1}c_{4l + 3} + d_{4l + 3})\prod\limits_{\mathclap{k_2 = l + 1}}^{s}a_{4k_2 + 1}c_{4k_2 + 3} \right)}{  a_0\prod\limits_{\mathclap{k_1 = 0}}^{s} a_{4k_1 + 4}c_{4k_1 + 2} + \left(b_0\prod\limits_{\mathclap{k_1 = 0}}^{s} a_{4k_1 + 4}c_{4k_1 + 2}    +  \sum\limits_{l = 0}^{s}(a_{4l + 4}d_{4l + 2} + b_{4l + 4})\prod\limits_{\mathclap{k_2 = l + 1}}^{s}a_{4k_2 + 4}c_{4k_2 + 2} \right)x_{-2}y_{-1}}\right)\nonumber \\
           & \times \frac{x_{-2}^{n + 1}y_{-2}^{n}y_{-1}}{x_0^{n}y_0^{n+1}(a_0 + b_0x_{-2}y_{-1})},\label{21a3p}
\end{align}
\begin{align}
y_{4n + 1} & = \frac{y_{1}x_{-2}^{n}}{x_2^{n}}\prod_{s = 0}^{n - 1}\frac{\prod\limits_{k_1 = 0}^{s - 1} a_{4k_1 + 5}c_{4k_1 + 3} + x_2y_1\sum\limits_{l = 0}^{s - 1}\left((a_{4l + 5}d_{4l + 3} + b_{4l + 5})\prod\limits_{k_2 = l + 1}^{s - 1}a_{4k_2 + 5}c_{4k_2 + 3} \right)}{\prod\limits_{k_1 = 0}^{s} a_{4k_1}c_{4k_1 + 2} + x_{-2}y_{-1}\sum\limits_{l = 0}^{s}\left((b_{4l}c_{4l + 2} + d_{4l  + 2})\prod\limits_{k_2 = l + 1}^{s}a_{4k_2 }c_{4k_2 + 2} \right)} \nonumber \\ 
          & \quad \times \frac{\prod\limits_{k_1 = 0}^{s} a_{4k_1 + 3}c_{4k_1 + 1} + x_0y_{-1}\sum\limits_{l = 0}^{s}\left((a_{4l + 3}d_{4l + 1} + b_{4l + 3})\prod\limits_{k_2 = l + 1}^{s}a_{4k_2 + 3}c_{4k_2 + 1} \right)}{\prod\limits_{k_1 = 0}^{s} a_{4k_1 + 2}c_{4k_1 + 4} + x_0y_1\sum\limits_{l = 0}^{s}\left((b_{4l + 2}c_{4l + 4} + d_{4l + 4})\prod\limits_{k_2 = l + 1}^{s}a_{4k_2 + 2}c_{4k_2 + 4} \right)}\nonumber\\
           & = \frac{y_{1}x_{-2}^{n}}{x_2^{n}}\prod_{s = 0}^{n - 1}\frac{\prod\limits_{k_1 = 0}^{s - 1} a_{4k_1 + 5}c_{4k_1 + 3} + \frac{x_{-1}y_0}{a_1 + b_1x_{-1}y_0}\sum\limits_{l = 0}^{s - 1}\left((a_{4l + 5}d_{4l + 3} + b_{4l + 5})\prod\limits_{k_2 = l + 1}^{s - 1}a_{4k_2 + 5}c_{4k_2 + 3} \right)}{\prod\limits_{k_1 = 0}^{s} a_{4k_1}c_{4k_1 + 2} + x_{-2}y_{-1}\sum\limits_{l = 0}^{s}\left((b_{4l}c_{4l + 2} + d_{4l  + 2})\prod\limits_{k_2 = l + 1}^{s}a_{4k_2 }c_{4k_2 + 2} \right)} \nonumber \\
           & \quad \times \frac{\prod\limits_{k_1 = 0}^{s} a_{4k_1 + 3}c_{4k_1 + 1} + x_0y_{-1}\sum\limits_{l = 0}^{s}\left((a_{4l + 3}d_{4l + 1} + b_{4l + 3})\prod\limits_{k_2 = l + 1}^{s}a_{4k_2 + 3}c_{4k_2 + 1} \right)}{\prod\limits_{k_1 = 0}^{s} a_{4k_1 + 2}c_{4k_1 + 4} + \frac{x_{-1}y_{-2}}{c_0 + d_0x_{-1}y_{-2}}\sum\limits_{l = 0}^{s}\left((b_{4l + 2}c_{4l + 4} + d_{4l + 4})\prod\limits_{k_2 = l + 1}^{s}a_{4k_2 + 2}c_{4k_2 + 4} \right)}\nonumber\\
           & = \prod_{s = 0}^{n - 1}\left(\frac{    a_1\prod\limits_{\mathclap{k_1 = 0}}^{\mathclap{s - 1}} a_{4k_1 + 5}c_{4k_1 + 3} + \left(b_1\prod\limits_{\mathclap{k_1 = 0}}^{\mathclap{s - 1}} a_{4k_1 + 5}c_{4k_1 + 3}  + \sum\limits_{l = 0}^{s - 1}(a_{4l + 5}d_{4l + 3} + b_{4l + 5})\prod\limits_{\mathclap{k_2 = l + 1}}^{\mathclap{s - 1}}a_{4k_2 + 5}c_{4k_2 + 3} \right)x_{-1}y_0}{\prod\limits_{\mathclap{k_1 = 0}}^{s} a_{4k_1}c_{4k_1 + 2} + x_{-2}y_{-1}\sum\limits_{l = 0}^{s}\left((b_{4l}c_{4l + 2} + d_{4l  + 2})\prod\limits_{\mathclap{k_2 = l + 1}}^{s}a_{4k_2 }c_{4k_2 + 2} \right)}\right. \nonumber \\
           & \left.\times \frac{\prod\limits_{\mathclap{k_1 = 0}}^{s} a_{4k_1 + 3}c_{4k_1 + 1} + x_0y_{-1}\sum\limits_{l = 0}^{s}\left((a_{4l + 3}d_{4l + 1} + b_{4l + 3})\prod\limits_{\mathclap{k_2 = l + 1}}^{s}a_{4k_2 + 3}c_{4k_2 + 1} \right)}{ c_0\prod\limits_{\mathclap{k_1 = 0}}^{s} a_{4k_1 + 2}c_{4k_1 + 4} + \left(d_0\prod\limits_{\mathclap{k_1 = 0}}^{s} a_{4k_1 + 2}c_{4k_1 + 4}    +  \sum\limits_{l = 0}^{s}(b_{4l + 2}c_{4l + 4} + d_{4l + 4})\prod\limits_{\mathclap{k_2 = l + 1}}^{s}a_{4k_2 + 2}c_{4k_2 + 4} \right)x_{-1}y_{-2}}\right)\nonumber \\
           & \times \frac{x_{-2}^{n}y_{-2}^{n + 1}x_{-1}}{x_0^{n +1}y_{0}^{n}(c_0 + d_0x_{-1}y_{-2})}.\label{21b3p}
\end{align}
\endgroup
Thus the explicit solution $\{x_{n}\}_{n=-2}^{\infty}$, $\{y_{n}\}_{n=-2}^{\infty}$ to Equation \eqref{1.1}  is given by equations \eqref{21a0p}, \eqref{21b0p}, \eqref{21a1p}, \eqref{21b1p}, \eqref{21a2p}, \eqref{21b2p}, \eqref{21a3p} and \eqref{21b3p}. In the following section, we look at special cases, where solutions are expressed in terms of initial values. In some of these cases, we  generalise and simplify some results found in \cite{EI}.
\subsection{The case $a_n, b_n c_n, d_n$ are constant and explicit solutions}
Assume that $a_n = a, b_n = b, c_n = d, d_n = d$ are constants in the equations obtained in the previous section. The solution is then given by:
\begingroup\makeatletter\def\f@size{9}\check@mathfonts
\def\maketag@@@#1{\hbox{\m@th\large\normalfont#1}}%
\begin{align}\label{21a0}
x_{4n - 2} & = \frac{x_0^{n}y_0^{n}}{y_{-2}^nx_{-2}^{n - 1}}\prod_{s = 0}^{n - 1}\frac{(ac)^{s} + (bc + d)x_{-2}y_{-1}\sum\limits_{l = 0}^{s - 1}(ac)^{l}}{(ac)^{s} + (ad + b)x_0y_{-1}\sum\limits_{l = 0}^{s - 1}(ac)^{l}}\,\,\frac{ (a^{s}c^{s + 1} + \left((ac)^{s}d + (bc + d)\sum\limits_{l = 0}^{s - 1}(ac)^{l}\right)y_{-2}x_{-1} } { (a^{s + 1}c^{s} + \left((ac)^{s}b + (ad + b)\sum\limits_{l = 0}^{s - 1}(ac)^{l}\right)x_{-1}y_0},
\end{align}
\begin{align}\label{21a1}
x_{4n - 1}  & = \frac{x_{-1}x_{-2}^{n}y_{-2}^{n}}{x_0^{n}y_0^n}\prod_{s = 0}^{n - 1}\frac{(ac)^{s} + (bc + d)x_{-1}y_0\sum\limits_{l = 0}^{s - 1}(ac)^{l}}{a^{s + 1}c^{s} + \left((ac)^{s}b + (ad + b)\sum\limits_{l = 0}^{s- 1}(ac)^{l}\right)x_{-2}y_{-1}}\\
           & \times \frac{a^{s}c^{s + 1}  + \left((ac)^{s}d + (bc + d)\sum\limits_{l = 0}^{s - 1}(ac)^l\right)x_0y_{-1}}{(ac)^{s+1} + (ad + b)x_{-1}y_{-2}\sum\limits_{l = 0}^{s}(ac)^l},
\end{align}
\begin{align}\label{21a2}
x_{4n} & = \frac{x_{0}^{n + 1}y_0^{n}}{x_{-2}^{n}y_{-2}^n}\prod_{s = 0}^{n - 1}\frac{(a^{s}c^{s+1} + \left((ac)^{s}d + (bc + d)\sum\limits_{l = 0}^{s - 1}(ac)^{l}\right)x_{-1}y_{-2}}{(a^{s + 1}c^s + \left((ac)^{s}b + (ad + b)\sum\limits_{l = 0}^{s - 1}(ac)^{l}\right)x_{-1}y_0}\,\, \frac{(ac)^{s + 1} + (bc + d)x_{-2}y_{-1}\sum\limits_{l = 0}^{s}(ac)^l}{(ac)^{s + 1} + (ad + b)x_0y_{-1}\sum\limits_{l = 0}^{s}(ac)^{l}},
\end{align}
\begin{align}\label{21a3}
x_{4n + 1}
           & =  \frac{x_{-2}^{n +1}y_{-2}^{n}y_{-1}}{x_0^{n}y_{0}^{n +1}(a + bx_{-2}y_{-1})}\prod_{s = 0}^{n - 1}\left(\frac{a^{s}c^{s+1} + \left((ac)^{s}d + (bc + d)\sum\limits_{l = 0}^{s - 1}(ac)^{l}\right)x_0y_{-1}}{(ac)^{s + 1} + (ad + b)x_{-1}y_{-2}\sum\limits_{l = 0}^{s}(ac)^{l}}\right. \nonumber \\
           &\quad \quad \quad \quad \quad \quad \quad \quad \quad \quad \times \, \left.\frac{(ac)^{s + 1} + (bc + d)x_{-1}y_0\sum\limits_{l = 0}^{s}(ac)^l}{ a^{s + 2}c^{s + 1} + \left((ac)^{s + 1}b + (ad + b)\sum\limits_{l = 0}^{s}(ac)^{l}\right)x_{-2}y_{-1}}\right),
\end{align}
\begin{align}\label{21b0}
y_{4n - 2}
          & = \frac{x_0^{n}y_0^{n}}{x_{-2}^{n}y_{-2}^{n - 1}}\prod_{s = 0}^{n - 1} \frac{ (ac)^{s}  + (ad + b)x_{-1}y_{-2}\sum\limits_{l = 0}^{s - 1}(ac)^{l} }{ (ac)^{s} + (bc + d)x_{-1}y_0\sum\limits_{l = 0}^{s - 1} (ac)^{l}}\,\,\frac{a^{s + 1}c^{s} + \left((ac)^{s}b + (ad + b)\sum\limits_{l = 0}^{s- 1}(ac)^{l}\right)x_{-2}y_{-1}}{(a^{s}c^{s + 1} + \left((ac)^{s}d + (bc + d)\sum\limits_{l = 0}^{s - 1}(ac)^{l}\right)x_0y_{-1}},
\end{align}
\begin{align}\label{21b1}
y_{4n - 1}            & = \frac{x_{-2}^{n}y_{-2}^{n}y_{-1}}{x_0^{n}y_0^{n}}\prod_{i = 0}^{n - 1}\frac{(ac)^s + (ad + b)x_0y_{-1}\sum\limits_{l = 0}^{s - 1}(ac)^l}{a^{s}c^{s+1} + \left((ac)^{s}d + (bc + d)\sum\limits_{l = 0}^{s - 1}(ac)^l\right)x_{-1}y_{-2}}
 \frac{a^{s + 1}c^s + \left((ac)^{s}b + (ad + b)\sum\limits_{l = 0}^{s - 1}(ac)^l\right)x_{-1}y_0}{(ac)^{s + 1} + (bc + d)x_{-2}y_{-1}\sum\limits_{l = 0}^{s}(ac)^{l}},
\end{align}
\begin{align}\label{21b2}
y_{4n}       & = \frac{x_0^{n}y_{0}^{n + 1}}{x_{-2}^{n}y_{-2}^{n}}\prod_{s = 0}^{n - 1}\frac{(a^{s + 1}c^s + \left( (ac)^{s}b + (ad + b)\sum\limits_{l = 0}^{s- 1}(ac)^l\right)x_{-2}y_{-1}}{\prod\limits_{k_1 = 0}^{s - 1} a^{s}c^{s+1} + \left((ac)^{s}d + (bc + d)\sum\limits_{l = 0}^{s - 1}(ac)^l\right)x_0y_{-1}} \,\,\frac{(ac)^{s + 1} + (ad + b)x_{-1}y_{-2}\sum\limits_{l = 0}^{s}(ac)^l}{(ac)^{s + 1} + (bc + d)x_{-1}y_0\sum\limits_{l = 0}^{s}(ac)^{l}},
\end{align}
\begin{align}\label{21b3}
y_{4n + 1}          & = \frac{x_{-1}x_{-2}^{n}y_{-2}^{n+1}}{x_0^{n +1}y_0^{n}(c + dx_{-1}y_{-2})}\prod_{s = 0}^{n - 1}\left(\frac{ a^{s + 1}c^s + \left((ac)^{s}b + (ad + b)\sum\limits_{l = 0}^{s - 1}(ac)^l\right)x_{-1}y_0}{(ac)^{s + 1} + (bc + d)x_{-2}y_{-1}\sum\limits_{l = 0}^{s}(ac)^{l}}\right.\\
          & \quad \quad \quad \quad \quad \quad \quad \quad \quad \quad \times \, \left.\frac{(ac)^{s + 1} + (ad + b)x_0y_{-1}\sum\limits_{l = 0}^{s}(ac)^{l}}{a^{s + 1}c^{s + 2} + \left((ac)^{s + 1}d + (bc + d)\sum\limits_{l = 0}^{s}(ac)^l\right)x_{-1}y_{-2}}\right).
\end{align}
\endgroup
\subsubsection{The case $a = b = c = d = 1$}
The solution of the system, which is Theorem  1 of Elsayed \cite{EI} is:
\begingroup\makeatletter\def\f@size{9}\check@mathfonts
\def\maketag@@@#1{\hbox{\m@th\large\normalfont#1}}%
\begin{align}
x_{4n - 2} & =  \frac{x_0^{n}y_0^{n}}{y_{-2}^nx_{-2}^{n - 1}}\prod_{s = 0}^{n - 1}\frac{1 + 2sx_{-2}y_{-1}}{1 + 2sx_0y_{-1}}\,\,\frac{ 1 + (1 + 2s)y_{-2}x_{-1} } { 1 + (1 + 2s)x_{-1}y_0}, \label{21a0}\\
x_{4n - 1} & =  \frac{x_{-1}x_{-2}^{n}y_{-2}^{n}}{x_0^{n}y_0^n}\prod_{s = 0}^{n - 1}\frac{1 + 2sx_{-1}y_0}{1 + (1 + 2s)x_{-2}y_{-1}}\,\,\frac{1  + (1 + 2s)x_0y_{-1}}{1 + (2s + 2)x_{-1}y_{-2}} ,\label{21a1}\\
x_{4n} & = \frac{x_{0}^{n + 1}y_0^{n}}{x_{-2}^{n}y_{-2}^n}\prod_{s = 0}^{n - 1}\frac{1 + (1 + 2s)x_{-1}y_{-2}}{1 + (1 + 2s)x_{-1}y_0}\,\, \frac{1 + (2s+2)x_{-2}y_{-1}}{1 + (2s + 2)x_0y_{-1}}    ,\label{21a2}\\
x_{4n + 1} & = \frac{x_{-2}^{n +1}y_{-2}^{n}y_{-1}}{x_0^{n}y_{0}^{n +1}(1 + x_{-2}y_{-1})}\prod_{s = 0}^{n - 1}\frac{1 + (1 + 2s)x_0y_{-1}}{1 + (2s + 2)x_{-1}y_{-2}}\frac{1 + (2s + 2)x_{-1}y_0}{ 1 + (2s + 3)x_{-2}y_{-1}}   ,\label{21a3}\\
y_{4n - 2} & = \frac{x_0^{n}y_0^{n}}{x_{-2}^{n}y_{-2}^{n - 1}}\prod_{s = 0}^{n - 1} \frac{ 1  + 2sx_{-1}y_{-2}}{1 + 2sx_{-1}y_0}\,\,\frac{1 + (1 + 2s)x_{-2}y_{-1}}{1 + (1 + 2s)x_0y_{-1}}   ,\label{21b0}\\
y_{4n - 1}  & = \frac{x_{-2}^{n}y_{-2}^{n}y_{-1}}{x_0^{n}y_0^{n}}\prod_{i = 0}^{n - 1}\frac{1 + 2sx_0y_{-1}}{1 + (1 + 2s)x_{-1}y_{-2}}\,\,\frac{1 + (1 + 2s)x_{-1}y_0}{1 + (2s + 2)x_{-2}y_{-1}}     ,\label{21b1}\\
y_{4n} & = \frac{x_0^{n}y_{0}^{n + 1}}{x_{-2}^{n}y_{-2}^{n}}\prod_{s = 0}^{n - 1}\frac{1 + (1 + 2s)x_{-2}y_{-1}}{1 + (1 + 2s)x_0y_{-1}} \,\,\frac{1 + (2s + 2)x_{-1}y_{-2}}{1+ (2s + 2)x_{-1}y_0}  ,\label{21b2}\\
y_{4n + 1} & =  \frac{x_{-1}x_{-2}^{n}y_{-2}^{n+1}}{x_0^{n +1}y_0^{n}(1 + x_{-1}y_{-2})}\prod_{s = 0}^{n - 1}\frac{ 1 + (1 + 2s)x_{-1}y_0}{1 + (2s + 2)x_{-2}y_{-1}}\,\,\frac{1 + (2s + 2)x_0y_{-1}}{1 + (2s + 3)x_{-1}y_{-2}}.\label{21b3}
\end{align}
\endgroup
\subsubsection{The case $a = c = 1, b = d = -1$}
In this case, we obtain Theorem 2 of Elsayed \cite{EI} as follows
\begingroup\makeatletter\def\f@size{9}\check@mathfonts
\def\maketag@@@#1{\hbox{\m@th\large\normalfont#1}}%
\begin{align}
x_{4n - 2}  & = \frac{x_0^{n}y_0^{n}}{y_{-2}^nx_{-2}^{n - 1}}\prod_{s = 0}^{n - 1}\frac{1 - 2sx_{-2}y_{-1}}{1 - 2sx_0y_{-1}}\,\,\frac{ 1 -(1 + 2s)y_{-2}x_{-1} } { 1 - (1 + 2s)x_{-1}y_0},\label{22a0}\\
x_{4n - 1}  & = \frac{x_{-1}x_{-2}^{n}y_{-2}^{n}}{x_0^{n}y_0^n}\prod_{s = 0}^{n - 1}\frac{1 -2sx_{-1}y_0}{1 - (1 + 2s)x_{-2}y_{-1}}\,\,\frac{1  - (1 + 2s)x_0y_{-1}}{1 - (2s + 2)x_{-1}y_{-2}},\label{22a1}\\
x_{4n} & = \frac{x_{0}^{n + 1}y_0^{n}}{x_{-2}^{n}y_{-2}^n}\prod_{s = 0}^{n - 1}\frac{1 - (1+2s)x_{-1}y_{-2}}{1 - (1 + 2s)x_{-1}y_0}\,\, \frac{1 - (2s + 2)x_{-2}y_{-1}}{1 - (2s + 2)x_0y_{-1}},\label{22a2}\\
x_{4n + 1} & =  \frac{x_{-2}^{n +1}y_{-2}^{n}y_{-1}}{x_0^{n}y_{0}^{n +1}(1 - x_{-2}y_{-1})}\prod_{s = 0}^{n - 1}\frac{1 - (1 + 2s)x_0y_{-1}}{1 - (2s + 2)x_{-1}y_{-2}}\,\,\frac{1 - (2s + 2)x_{-1}y_0}{ 1 - (2s + 3)x_{-2}y_{-1}},\label{22a3}\\
y_{4n - 2} & = \frac{x_0^{n}y_0^{n}}{x_{-2}^{n}y_{-2}^{n - 1}}\prod_{s = 0}^{n - 1} \frac{ 1  - 2sx_{-1}y_{-2} }{ 1 - 2sx_{-1}y_0}\,\,\frac{1 - (1 + 2s)x_{-2}y_{-1}}{1 - (1 + 2s)x_0y_{-1}},\label{22b0}
\end{align}
\begin{align}
y_{4n - 1}  & = \frac{x_{-2}^{n}y_{-2}^{n}y_{-1}}{x_0^{n}y_0^{n}}\prod_{i = 0}^{n - 1}\frac{1 - 2sx_0y_{-1}}{1 - (1 + 2s)x_{-1}y_{-2}}\,\,\frac{1 - (1 + 2s)x_{-1}y_0}{1 - (2s + 2)x_{-2}y_{-1}},\label{22b1}\\
y_{4n}  & = \frac{x_0^{n}y_{0}^{n + 1}}{x_{-2}^{n}y_{-2}^{n}}\prod_{s = 0}^{n - 1}\frac{1 - (1 + 2s)x_{-2}y_{-1}}{1 - (1 + 2s)x_0y_{-1}} \,\,\frac{1 - (2s + 2)x_{-1}y_{-2}}{1 - (2s + 2)x_{-1}y_0},\label{22b2}\\
y_{4n + 1} & = \frac{x_{-1}x_{-2}^{n}y_{-2}^{n+1}}{x_0^{n +1}y_0^{n}(1 - x_{-1}y_{-2})}\prod_{s = 0}^{n - 1}\frac{ 1 - (1 + 2s)x_{-1}y_0}{1 - (2s + 2)x_{-2}y_{-1}}\,\,\frac{1 - (2s + 2)x_0y_{-1}}{1 - (2s + 3)x_{-1}y_{-2}}.\label{22b3}
\end{align}
\endgroup
\subsubsection{Some of cases where the constants are unit}
Substituting the following values, we obtain the solution that Elsayed \cite{EI} got. \\
$a = c = d = -1, b = 1$ (Theorem 3 in \cite{EI})
$a = b = c = -1, d = 1$ (Theorem 4 in \cite{EI})
$a = b = c = 1, d = -1$ (Theorem 13 in \cite{EI})
$a = c = d = 1, b = -1$ (Theorem 14 in \cite{EI})
$a = c = -1, b = d = 1 $(Theorem 15 in \cite{EI})
$a = b = c = d = -1 $ (Theorem 16 in \cite{EI})
\subsubsection{Remaining cases where the constants are unit}
For each of the following cases:\\
$b = c = d = 1, a = -1$,\\
$a = b = d = -1, c = 1$,\\
$a = b = 1, \,\,c = d = -1$,\\
$a = d = 1, b = c = -1$, \\
 $a = b = d = 1,\,\, c = -1$, \\
 $a = b = -1,\,\, c = d = 1$, \\
 $b = c = d = -1,\,\, a = 1$, \\
$a = d = -1,\,\, b = c = 1$, \\
our solution  is represented by 8 equations, whereas in Elsayed's case \cite{EI} (see Theorems 5,6,7,8,9,10,11,12), the solution is represented by 16 equations. Thus ours is a great simplification of Elsayed's solution.
\subsection{The case when $a_n, b_n, c_n, d_n$ are sequences of period 4}
In this setting, we the solution is given by:
\begingroup\makeatletter\def\f@size{9}\check@mathfonts
\def\maketag@@@#1{\hbox{\m@th\large\normalfont#1}}%
\begin{align}\label{21a0}
x_{4n - 2} & = \frac{x_0^{n}y_0^{n}}{x_{-2}^{n - 1}y_{-2}^{n}}\prod_{s = 0}^{n - 1}\frac{(a_{0}c_{2})^{s} + (b_{0}c_{2} + d_{2})x_{-2}y_{-1}\sum\limits_{l = 0}^{s - 1}(a_{0 }c_{2})^{l} }{(a_{3}c_{1})^{s} + (a_{3}d_{1} + b_{3})x_0y_{-1}\sum\limits_{l = 0}^{s - 1}(a_{3}c_{1})^{l}} \nonumber\\
       & \quad \times \frac{    c_{0}(a_{2}c_{0})^{s}  + \left(d_{0}(a_{2}c_{0})^{s} + (b_{2}c_{0} + d_{0})\sum\limits_{l = 0}^{s - 1}(a_{2}c_{0})^{l} \right)x_{-1}y_{-2}  }{  a_{1}(a_{1}c_{3})^{s} + \left( b_{1}(a_{1}c_{3})^{s}  + (a_{1}d_{3} + b_{1})\sum\limits_{l = 0}^{s - 1}(a_{1}c_{3})^{l} \right)x_{-1}y_0},
\end{align}
\begin{align}
y_{4n - 2} & = \frac{x_0^{n}y_0^{n}}{x_{-2}^{n}y_{-2}^{n - 1}}\prod_{s = 0}^{n - 1}\frac{(a_{2}c_{0})^{s} + (a_{2}d_{0} + b_{2})x_{-1}y_{-2}\sum\limits_{l = 0}^{s - 1}(a_{2}c_{0})^{l}}{(a_{1}c_{3})^{s} + (b_{1}c_{3} + d_{3})x_{-1}y_0\sum\limits_{l = 0}^{s - 1}(a_{1}c_{3})^{l}}\nonumber\\
& \quad \times \frac{   a_{0}(a_{0}c_{2})^{l} + \left(b_{0}(a_{0}c_{2})^{s} + (a_{0}d_{2} + b_{0})\sum\limits_{l = 0}^{s- 1}(a_{0}c_{2})^{l} \right)x_{-2}y_{-1}}{ c_{1}(a_{3}c_{1})^{s} +  \left( d_{1}(a_{3}c_{1})^{s}  + (b_{3}c_{1} + d_{1})\sum\limits_{l = 0}^{s - 1}(a_{3}c_{1})^{l} \right)x_{0}y_{-1}},\label{21b0}\\
x_{4n - 1}  & = \frac{x_{-1}x_{-2}^{n}y_{-2}^n}{x_0^{n}y_0^{n}}\prod_{s = 0}^{n - 1}\frac{\prod\limits_{k_1 = 0}^{s - 1} a_{1}c_{3} + x_{-1}y_0\sum\limits_{l = 0}^{s - 1}\left((b_{1}c_{3} + d_{3})\prod\limits_{k_2 = l + 1}^{s - 1}a_{1}c_{3} \right)}{ a_{0}\prod\limits_{k_1 = 0}^{s - 1} a_{0}c_{2} + \left(b_{0}\prod\limits_{k_1 = 0}^{s - 1} a_{0}c_{2}    +  \sum\limits_{l = 0}^{s- 1}(a_{0}d_{2} + b_{0})\prod\limits_{k_2 = l + 1}^{s - 1}a_{0}c_{2} \right)x_{-2}y_{-1}}\nonumber \\
           & \quad \times \frac{   c_1(a_{3}c_{1})^{s} +  \left(d_1(a_{3}c_{1})^{s}  + (b_{3}c_{1} + d_{1})\sum\limits_{l = 0}^{s - 1}(a_{3}c_{1})^{l} \right)x_0y_{-1}}{(a_{2}c_{0})^{s  + 1} + (a_{2}d_{0} + b_{2})x_{-1}y_{-2}\sum\limits_{l = 0}^{s}(a_{2}c_{0})^{l} },\label{21a1}\\
y_{4n - 1} & = \frac{y_{-1}x_{-2}^{n}y_{-2}^n}{x_0^{n}y_0^{n}}\prod_{i = 0}^{n - 1}\frac{(a_{3}c_{1})^{s} + (a_{3}d_{1} + b_{3})x_0y_{-1}\sum\limits_{l = 0}^{s - 1}(a_{3}c_{1})^{l}}{   c_0(a_{2}c_{0})^{s} + \left( d_0(a_{2}c_{0})^{s}    +  (b_{2}c_{0} + d_{0})\sum\limits_{l = 0}^{s - 1}(a_{2}c_{0})^{l} \right)x_{-1}y_{-2}   } \nonumber \\
            & \quad \times \frac{  a_1(a_{1}c_{3})^{s} + \left(b_1(a_{1}c_{3})^{s}   + (a_{1}d_{3} + b_{1})\sum\limits_{l = 0}^{s - 1}(a_{1}c_{3})^{l} \right)x_{-1}y_0}{(a_{0}c_{2})^{s + 1} + (b_{4l}c_{2} + d_{2})x_{-2}y_{-1}\sum\limits_{l = 0}^{s}(a_{0}c_{2})^{l}},\label{21b1}\\
x_{4n} & = \frac{x_{0}^{n + 1}y_0^{n}}{x_{-2}^{n}y_{-2}^{n}}\prod_{s = 0}^{n - 1}\frac{   c_0(a_{2}c_{0})^{s} + \left(d_0(a_{2}c_{0})^{s}  + (b_{2}c_{0} + d_{0})\sum\limits_{l = 0}^{s - 1}(a_{2}c_{0})^{l} \right)x_{-1}y_{-2}}{ a_1(a_{1}c_{3})^{s} + \left(b_1(a_{1}c_{3})^{s}  + (a_{1}d_{3} + b_{1 })\sum\limits_{l = 0}^{s - 1}(a_{1}c_{3})^{l} \right)x_{-1}y_0} \nonumber\\
          & \quad \times \frac{(a_{0}c_{2})^{s + 1} + (b_{0}c_{2} + d_{2})x_{-2}y_{-1}\sum\limits_{l = 0}^{s}(a_{0}c_{2})^{l} }{(a_{3}c_{1})^{s + 1} + (a_{3}d_{1} + b_{3})x_0y_{-1}\sum\limits_{l = 0}^{s}(a_{3}c_{1})^{l}},\label{21a2}\\
y_{4n} & = \frac{x_0^{n}y_{0}^{n + 1}}{x_{-2}^{n}y_{-2}^{n}}\prod_{s = 0}^{n - 1}\frac{ a_0(a_{0}c_{2})^{s} + \left(b_0(a_{0}c_{2})^{s}   + (a_{0}d_{2} + b_{0})\sum\limits_{l = 0}^{s- 1}(a_{0}c_{2})^{l} \right)x_{-2}y_{-1} }{ c_1(a_{3}c_{1})^{s} + \left(d_1(a_{3}c_{1})^{s}   + (b_{3}c_{1} + d_{1 })\sum\limits_{l = 0}^{s - 1}(a_{3}c_{1})^{l} \right)x_0y_{-1}} \nonumber \\
          & \quad \times \frac{(a_{2}c_{0})^{s + 1} + (a_{2}d_{0} + b_{2})x_{-1}y_{-2}\sum\limits_{l = 0}^{s}(a_{2}c_{0})^{l}}{ (a_{1}c_{3})^{s + 1} + (b_{1}c_{3} + d_{3})x_{-1}y_0\sum\limits_{l = 0}^{s}(a_{1}c_{3})^{l}},\label{21b2}\\
x_{4n + 1} &  \frac{x_{-2}^{n + 1}y_{-2}^{n}y_{-1}}{x_0^{n}y_0^{n+1}(a_0 + b_0x_{-2}y_{-1})}\prod_{s = 0}^{n - 1}\frac{  c_1(a_{3}c_{1})^{s} + \left(d_1(a_{3}c_{1})^{s}    + (b_{3}c_{1} + d_{1})\sum\limits_{l = 0}^{s - 1}(a_{3}c_{1})^{l} \right)x_0y_{-1}}{(a_{2}c_{0})^{s + 1} + (a_{2}d_{4l} + b_{2})x_{-1}y_{-2}\sum\limits_{l = 0}^{s}(a_{2}c_{0})^{l} } \nonumber\\
          & \quad \times  \frac{(a_{1}c_{3})^{s + 1} + (b_{1}c_{3} + d_{3})x_{-1}y_0\sum\limits_{l = 0}^{s}(a_{1}c_{3})^{l}}{  a_0(a_{0}c_{2})^{s + 1} + \left(b_0(a_{0}c_{2})^{s + 1}    + (a_{0}d_{2} + b_{0})\sum\limits_{l = 0}^{s}(a_{0}c_{2})^{l} \right)x_{-2}y_{-1}},\label{21a3}
\end{align}
\begin{align}\label{21b3}
y_{4n + 1} &  = \frac{x_{-2}^{n}y_{-2}^{n + 1}x_{-1}}{x_0^{n +1}y_{0}^{n}(c_0 + d_0x_{-1}y_{-2})}\prod_{s = 0}^{n - 1}\frac{ a_1(a_{1}c_{3})^{s} + \left(b_1(a_{1}c_{3})^{s}  + (a_{1}d_{3} + b_{1})\sum\limits_{l = 0}^{s - 1}(a_{1}c_{3})^{l} \right)x_{-1}y_0}{(a_{0}c_{2})^{s + 1} + (b_{0}c_{2} + d_{2})x_{-2}y_{-1}\sum\limits_{l = 0}^{s}(a_{0}c_{2})^{l}} \nonumber \\
           & \quad \times \frac{(a_{3}c_{1})^{s + 1} + (a_{3}d_{1} + b_{3})x_0y_{-1}\sum\limits_{l = 0}^{s}(a_{3}c_{1})^{l}}{ c_0(a_{2}c_{0})^{s + 1} + \left(d_0(a_{2}c_{0})^{s + 1}    +  (b_{2}c_{0} + d_{0})\sum\limits_{l = 0}^{s}(a_{2}c_{0})^{l} \right)x_{-1}y_{-2}}.
\end{align}
\endgroup
\section{Existence of 2-periodic and 4-periodic solutions}
\begin{theorem}
If $x_{-2} = x_{0}, y_{-2} = y_{0}, a = c, b = d$ and $x_{-1}y_{-2} = x_{-2}y_{-1} = \frac{1 - a}{b}$, then the solution of the system $x_{n+1}=\frac{x_{n-2}y_{n-1}}{y_{n}(a + bx_{n-2}y_{n-1})}, \quad y_{n+1}=\frac{y_{n-2}x_{n-1}}{x_{n}(c + dy_{n-2}x_{n-1})}$ is periodic with period two.
\end{theorem}
\begin{proof}
Under the assumptions $b = d, a = c$, it is clear that
\begingroup\makeatletter\def\f@size{9}\check@mathfonts
\def\maketag@@@#1{\hbox{\m@th\large\normalfont#1}}%
\begin{align}
x_{4n - 2} & = \frac{x_0^{n}y_0^{n}}{y_{-2}^nx_{-2}^{n - 1}}\prod_{s = 0}^{n - 1}\frac{(ac)^{s} + (bc + d)x_{-2}y_{-1}\sum\limits_{l = 0}^{s - 1}(ac)^{l}}{(ac)^{s} + (ad + b)x_0y_{-1}\sum\limits_{l = 0}^{s - 1}(ac)^{l}}\,\,\frac{ (a^{s}c^{s + 1} + \left((ac)^{s}d + (bc + d)\sum\limits_{l = 0}^{s - 1}(ac)^{l}\right)y_{-2}x_{-1} } { (a^{s + 1}c^{s} + \left((ac)^{s}b + (ad + b)\sum\limits_{l = 0}^{s - 1}(ac)^{l}\right)x_{-1}y_0}, \nonumber\\
           & = x_{-2},\label{21a0}\\
x_{4n - 1}  & = \frac{x_{-1}x_{-2}^{n}y_{-2}^{n}}{x_0^{n}y_0^n}\prod_{s = 0}^{n - 1}\frac{(a)^{2s} + (ba + b)x_{-1}y_0\sum\limits_{l = 0}^{s - 1}(a)^{2l}}{a^{2s + 1} + \left((a)^{2s}b + (ab + b)\sum\limits_{l = 0}^{s- 1}(a)^{2l}\right)x_{-2}y_{-1}}
 \frac{a^{2s + 1}  + \left((a)^{2s}b + (ba + b)\sum\limits_{l = 0}^{s - 1}(a)^{2l}\right)x_0y_{-1}}{(a)^{2s+2} + (ab + b)x_{-1}y_{-2}\sum\limits_{l = 0}^{s}(a)^{2l}}, \nonumber\\
           & = x_{-1}\prod_{s = 0}^{n - 1}\frac{(a)^{2s} + (ba + b)x_{-1}y_{-2}\sum\limits_{l = 0}^{s - 1}(a)^{2l}}{(a)^{2s+2} + (ab + b)x_{-1}y_{-2}\sum\limits_{l = 0}^{s - 1}(a)^{2l} + (ab + b)x_{-1}y_{-2}a^{2s}}.\label{21a1}
\end{align}
\endgroup
But $x_{-1}y_{-2} = \frac{1 - a}{b}$ implies $a^{2} + (ab + b)x_{-1}y_{-2} = 1$ so that $a^{2s + 2} + (ab + b)x_{-1}y_{-2}a^{2s} = a^{2s}$. This yields $$ x_{4n - 1} = x_{-1}.$$
\begingroup\makeatletter\def\f@size{9}\check@mathfonts
\def\maketag@@@#1{\hbox{\m@th\large\normalfont#1}}%
\begin{align}\label{21a2}
x_{4n} & = \frac{x_{0}^{n + 1}y_0^{n}}{x_{-2}^{n}y_{-2}^n}\prod_{s = 0}^{n - 1}\frac{(a^{2s+1} + \left((a)^{2s}b + (ba + b)\sum\limits_{l = 0}^{s - 1}(a)^{2l}\right)x_{-1}y_{-2}}{(a^{2s + 1} + \left((a)^{2s}b + (ab + b)\sum\limits_{l = 0}^{s - 1}(a)^{2l}\right)x_{-1}y_0}\,\, \frac{(a)^{2s + 2} + (ba + b)x_{-2}y_{-1}\sum\limits_{l = 0}^{s}(a)^{2l}}{(a)^{2s + 2} + (ab + b)x_0y_{-1}\sum\limits_{l = 0}^{s}(a)^{2l}}\nonumber\\
      & = x_{0}
\end{align}
\begin{align}\label{21a3}
x_{4n + 1} & =  \frac{x_{-2}^{n +1}y_{-2}^{n}y_{-1}}{x_0^{n}y_{0}^{n +1}(a + bx_{-2}y_{-1})}\prod_{s = 0}^{n - 1}\left(\frac{a^{2s+1} + \left((a)^{2s}b + (ba + b)\sum\limits_{l = 0}^{s - 1}(a)^{2l}\right)x_0y_{-1}}{(a)^{2s + 2} + (ab + b)x_{-1}y_{-2}\sum\limits_{l = 0}^{s}(a)^{2l}}\right. \nonumber \\
           &\quad \quad \quad \quad \quad \quad \quad \quad \quad \quad \times \, \left.\frac{(a)^{2s + 2} + (ba + b)x_{-1}y_0\sum\limits_{l = 0}^{s}(a)^{2l}}{ a^{2s + 3} + \left((a)^{2s + 2}b + (ab + b)\sum\limits_{l = 0}^{s}(a)^{2l}\right)x_{-2}y_{-1}}\right), \nonumber\\
           &= \frac{x_0y_{-1}}{y_0(a + bx_{-2}y_{-1})}\prod_{s = 0}^{n - 1}\frac{ a^{2s+1} + \left((a)^{2s}b + (ba + b)\sum\limits_{l = 0}^{s - 1}(a)^{2l}\right)x_0y_{-1}}{ a^{2s + 3} + \left((a)^{2s + 2}b + (ab + b)\sum\limits_{l = 0}^{s}(a)^{2l}\right)x_{-2}y_{-1}} \nonumber\\
           & = x_{1}\prod_{s = 0}^{n - 1}\frac{ a^{2s+1} + \left((a)^{2s}b + (ba + b)\sum\limits_{l = 0}^{s - 1}(a)^{2l}\right)x_0y_{-1}}{ a^{2s + 3} + \left((a)^{2s + 2}b + (ab + b)\sum\limits_{l = 0}^{s}(a)^{2l}\right)x_{-2}y_{-1}}.
\end{align}
\endgroup
But $x_{-2}x_{-1} = \frac{1 - a}{b}$ implies that $a + bx_{0}x_{-1} = a^{3} + a^{2}bx_{-2}y_{-1} + (ab + b)x_{-2}y_{-1}$, which in turn yields $ a^{2s + 1} + a^{2s + 2}bx_{0}x_{-1} = a^{2s + 3} + a^{2s + 2}bx_{-2}y_{-1} + (ab + b)a^{2s}x_{-2}y_{-1}$. Thus
$$  x_{4n + 1} = x_{1}.$$
\noindent Similarly, it is not difficult to show that $y_{4n + j} = y_{j}$ for all $n \geq 0$ and $j\geq -2$. Since $x_{-2} = x_{0}$ and $y_{-2} = y_{0}$, we must have $x_{2 + j} = x_{j}$ and $y_{2 + j} =$ for all $j\geq -2$. Thus the solution has period 2.
\end{proof}
\noindent For illustration, we give numerical examples (See Figures \ref{graph2} and \ref{graph22}).
\begin{figure}[H]
    \centering
    \begin{minipage}{.47\textwidth}
        \centering
        \includegraphics[width=1\linewidth, height=0.2\textheight]{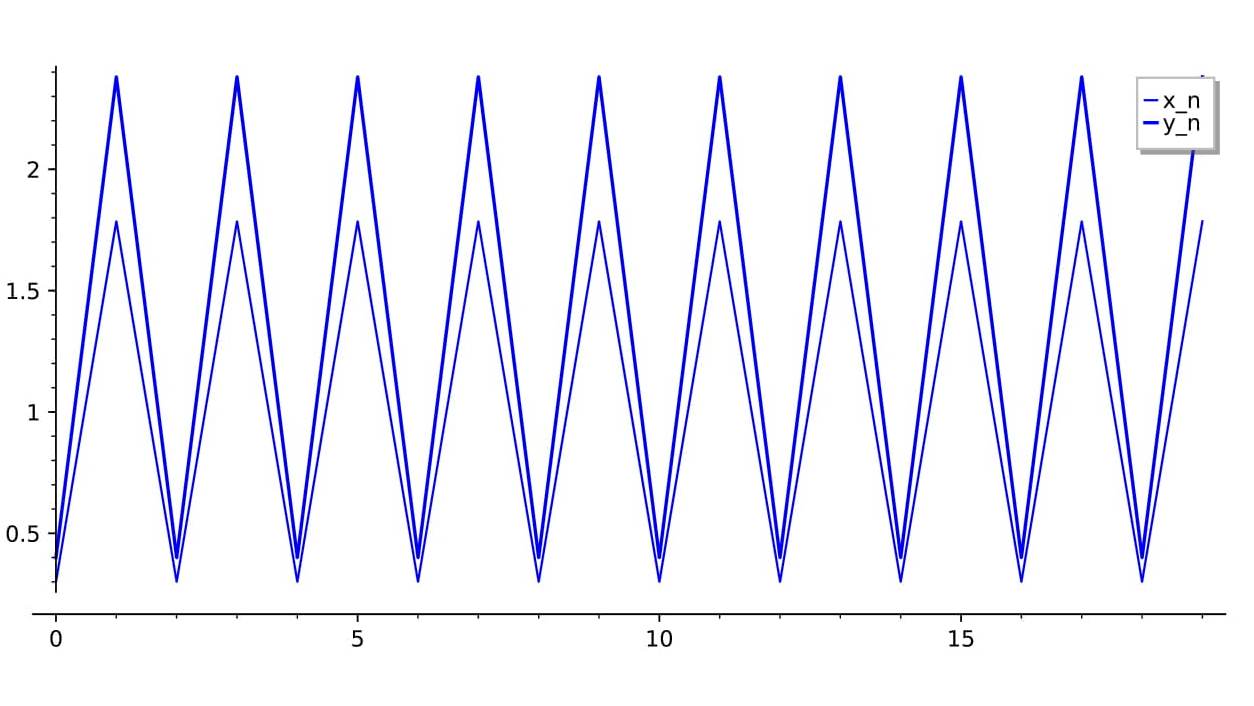}
        \caption{ {$a=, b=, c=, d=, x_{-3}=, x_{-2}=, x_{-1}=,  x_0=, y_{-3}=,\,y_{-2}=, y_{-1}=, y_0=$.}}
        \label{graph2}
    \end{minipage}\hfill
    \begin{minipage}{0.47\textwidth}
        \centering
        \includegraphics[width=1\linewidth, height=0.2\textheight]{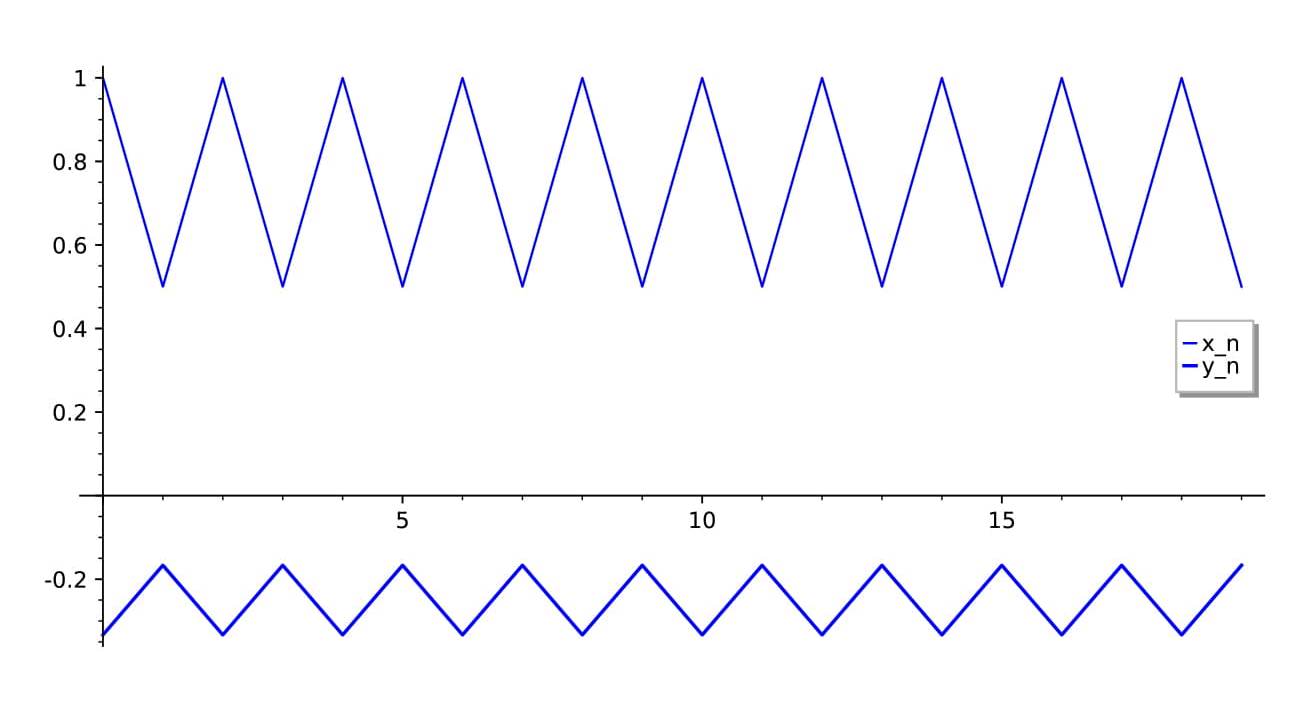}
        \caption{{$a=, b=, c=, d=, x_{-3}=, x_{-2}=, x_{-1}=,  x_0=, y_{-3}=,\,y_{-2}=, y_{-1}=, y_0=$.}}
        \label{graph22}
    \end{minipage}
\end{figure}
\begin{theorem}
If $x_0 = - x_{-2}, y_0 = -y_{-2}, a = c, b = -d$ and $x_{-1}y_{-2} = - x_{-2}y_{-1} = \frac{1 + a}{b}$, then the solution of the system $x_{n+1}=\frac{x_{n-2}y_{n-1}}{y_{n}(a + bx_{n-2}y_{n-1})}, \quad y_{n+1}=\frac{y_{n-2}x_{n-1}}{x_{n}(c + dy_{n-2}x_{n-1})}$ is periodic with period four.
\end{theorem}
\begin{proof}
Under the given assumptions $x_0 = - x_{-2}, y_0 = -y_{-2}, a = c, b = -d$,  we have
\begingroup\makeatletter\def\f@size{9}\check@mathfonts
\def\maketag@@@#1{\hbox{\m@th\large\normalfont#1}}%
\begin{align}
x_{4n - 2} & = \frac{x_0^{n}y_0^{n}}{y_{-2}^nx_{-2}^{n - 1}}\prod_{s = 0}^{n - 1}\frac{(a)^{2s} + (ab - b)x_{-2}y_{-1}\sum\limits_{l = 0}^{s - 1}(a)^{2l}}{(a)^{2s} +(ab - b)x_{-2}y_{-1}\sum\limits_{l = 0}^{s - 1}(a)^{2l}}\,\,\frac{ (a^{2s + 1} + \left(-(a)^{2s}b + (ab -b)\sum\limits_{l = 0}^{s - 1}(a)^{2l}\right)y_{-2}x_{-1} } { (a^{2s + 1} + \left((a)^{2s}b + (-ab + b)\sum\limits_{l = 0}^{s - 1}(ac)^{l}\right)x_{-1}y_0} \nonumber\\
          &  = \frac{x_0^{n}y_0^{n}}{y_{-2}^nx_{-2}^{n - 1}}\prod_{s = 0}^{n - 1}\frac{ (a^{2s + 1} + \left(-(a)^{2s}b + (ab -b)\sum\limits_{l = 0}^{s - 1}(a)^{2l}\right)-(y_{0})x_{-1} } { (a^{2s + 1} + \left((a)^{2s}b + (-ab + b)\sum\limits_{l = 0}^{s - 1}(ac)^{l}\right)x_{-1}y_0} \nonumber\\
          & = x_{-2},\label{21a0}\\
x_{4n - 1}  & = \frac{x_{-1}x_{-2}^{n}y_{-2}^{n}}{x_0^{n}y_0^n}\prod_{s = 0}^{n - 1}\left(\frac{(a)^{2s} + (ab - b)x_{-1}y_0\sum\limits_{l = 0}^{s - 1}(ac)^{l}}{a^{2s + 1} + \left((a)^{2s}b + (-ab + b)\sum\limits_{l = 0}^{s- 1}(a)^{2l}\right)x_{-2}y_{-1}}\right.\nonumber\\
           &\left. \times \frac{a^{2s + 1}  + \left(-(a)^{2s}b + (ab - b)\sum\limits_{l = 0}^{s - 1}(a)^{2l}\right)x_0y_{-1}}{(a)^{2s+2} + (-ab + b)x_{-1}y_{-2}\sum\limits_{l = 0}^{s}(a)^{2l}} \right) \nonumber\\
           & =\frac{x_{-1}x_{-2}^{n}y_{-2}^{n}}{x_0^{n}y_0^n}\prod_{s = 0}^{n - 1}\frac{(a)^{2s} + (ab - b)x_{-1}y_0\sum\limits_{l = 0}^{s - 1}(ac)^{l}}{(a)^{2s+2} + (ab - b)x_{-1}y_{0}\sum\limits_{l = 0}^{s}(a)^{2l}}\label{21a1}
\end{align}
\endgroup
But $x_{-1}y_{-2} = \frac{1 + a}{b}$, i.e. $x_{-1}y_{0} = - \frac{(1 + a)(1 - a)}{b(1 - a)}$ implies that
$(a - 1)bx_{-1}y_{0} = (1 - a)(1 + a)$, i.e. $a^{2} + (ab -b)x_{-1}y_{0} = 1$ implying
$$ a^{2s + 2}  + (ab  - b)x_{-1}y_{0}a^{2s} = a^{2s}$$ so that
$$x_{4n - 1} = x_{-1}.$$
\begingroup\makeatletter\def\f@size{9}\check@mathfonts
\def\maketag@@@#1{\hbox{\m@th\large\normalfont#1}}%
\begin{align}
x_{4n} & = \frac{x_{0}^{n + 1}y_0^{n}}{x_{-2}^{n}y_{-2}^n}\prod_{s = 0}^{n - 1}\frac{a^{2s+1} + \left(-(a)^{2s}b + (ab -b)\sum\limits_{l = 0}^{s - 1}(a)^{2l}\right)x_{-1}(-y_{0})}{(a^{2s + 1} + \left((a)^{2s}b + (-ab + b)\sum\limits_{l = 0}^{s - 1}(a)^{2l}\right)x_{-1}y_0}\,\, \frac{(a)^{2s + 2} + (ab  - b)x_{-2}y_{-1}\sum\limits_{l = 0}^{s}(a)^{2l}}{(a)^{2s + 2} + (-ab + b)x_0y_{-1}\sum\limits_{l = 0}^{s}(a)^{2l}} \nonumber\\
      & = x_0,\label{21a2}\\
x_{4n + 1}  & =  \frac{x_{-2}^{n +1}y_{-2}^{n}y_{-1}}{x_0^{n}y_{0}^{n +1}(a + bx_{-2}y_{-1})}\prod_{s = 0}^{n - 1}\left(\frac{a^{2s+1} + \left(-(a)^{2s}b + (ab - b)\sum\limits_{l = 0}^{s - 1}(a)^{2l}\right)x_0y_{-1}}{(a)^{2s + 2} + (-ab + b)x_{-1}y_{-2}\sum\limits_{l = 0}^{s}(a)^{2l}}\right. \nonumber \\
           &\quad \quad \quad \quad \quad \quad \quad \quad \quad \quad \times \, \left.\frac{(a)^{2s + 2} + (ab - b)x_{-1}y_0\sum\limits_{l = 0}^{s}(a)^{2l}}{ a^{2s + 3} + \left((a)^{2s + 2}b + (-ab + b)\sum\limits_{l = 0}^{s}(a)^{2l}\right)x_{-2}y_{-1}}\right)\nonumber\\
           & = \frac{x_{-2}^{n +1}y_{-2}^{n}y_{-1}}{x_0^{n}y_{0}^{n +1}(a + bx_{-2}y_{-1})}\prod_{s = 0}^{n - 1}\left(\frac{a^{2s+1} + \left((a)^{2s}b + (-ab + b)\sum\limits_{l = 0}^{s - 1}(a)^{2l}\right)x_{-2}y_{-1}}{a^{2s + 3} + \left((a)^{2s + 2}b + (-ab + b)\sum\limits_{l = 0}^{s}(a)^{2l}\right)x_{-2}y_{-1}}\right).\label{21a3}
\end{align}
\endgroup
However, $-x_{-2}y_{-1} = \frac{1 + a}{b}$ implies that $a^{3} + (a^{2}b - ab + b)x_{-2}y_{-1} = a + bx_{-2}y_{-1}$ which yields
$$ a^{2s + 3} + (a^{2s + 2}b + (-ab + b)a^{2s})x_{-2}y_{-1} = a^{2s + 1} + a^{2s}bx_{-2}y_{-1}$$ so that
$$ x_{4n + 1} = x_{1}.$$
\noindent Similarly, one can show that $y_{4n + j} = y_{j}$ for all $n \geq 0$ and $j \geq -2$. Indeed, the solution under the given assumptions is periodic with period four.
\end{proof}
\noindent For illustration, we give numerical examples (See Figures \ref{graph4} and \ref{graph44}).
\begin{figure}[H]
    \centering
    \begin{minipage}{.47\textwidth}
        \centering
        \includegraphics[width=1\linewidth, height=0.2\textheight]{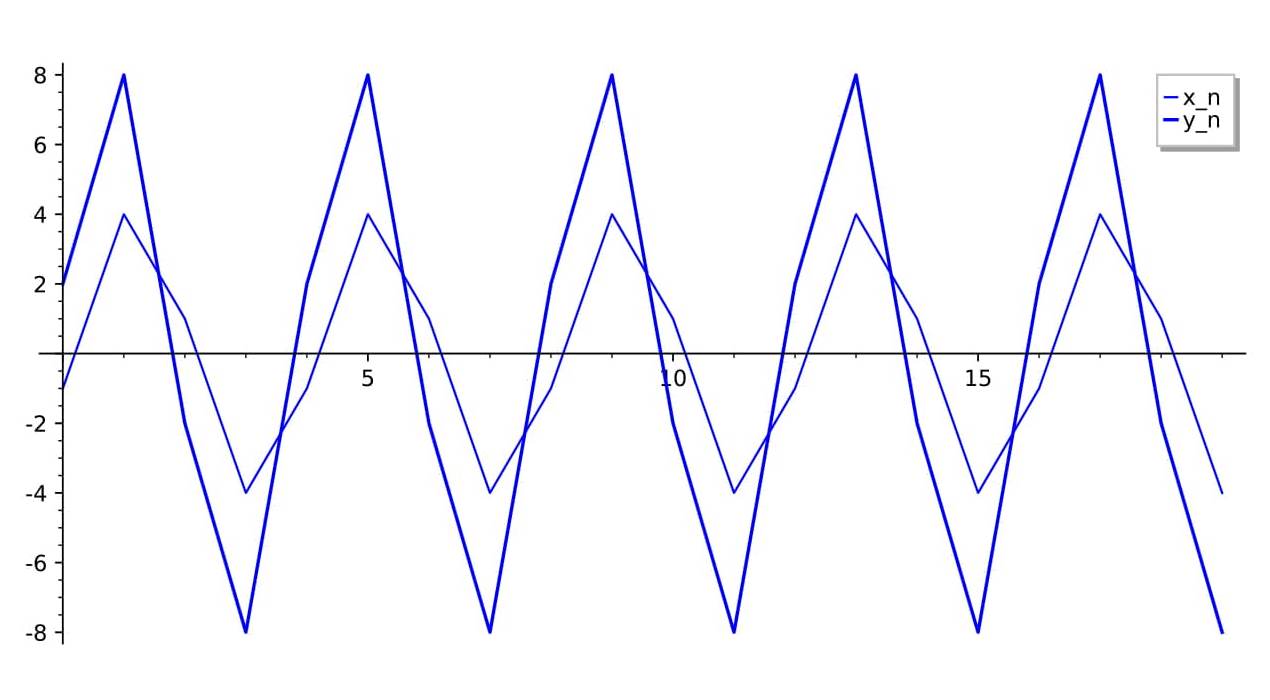}
        \caption{ {$a=, b=, c=, d=, x_{-3}=, x_{-2}=, x_{-1}=,  x_0=, y_{-3}=,\,y_{-2}=, y_{-1}=, y_0=$.}}
        \label{graph4}
    \end{minipage}\hfill
    \begin{minipage}{0.47\textwidth}
        \centering
        \includegraphics[width=1\linewidth, height=0.2\textheight]{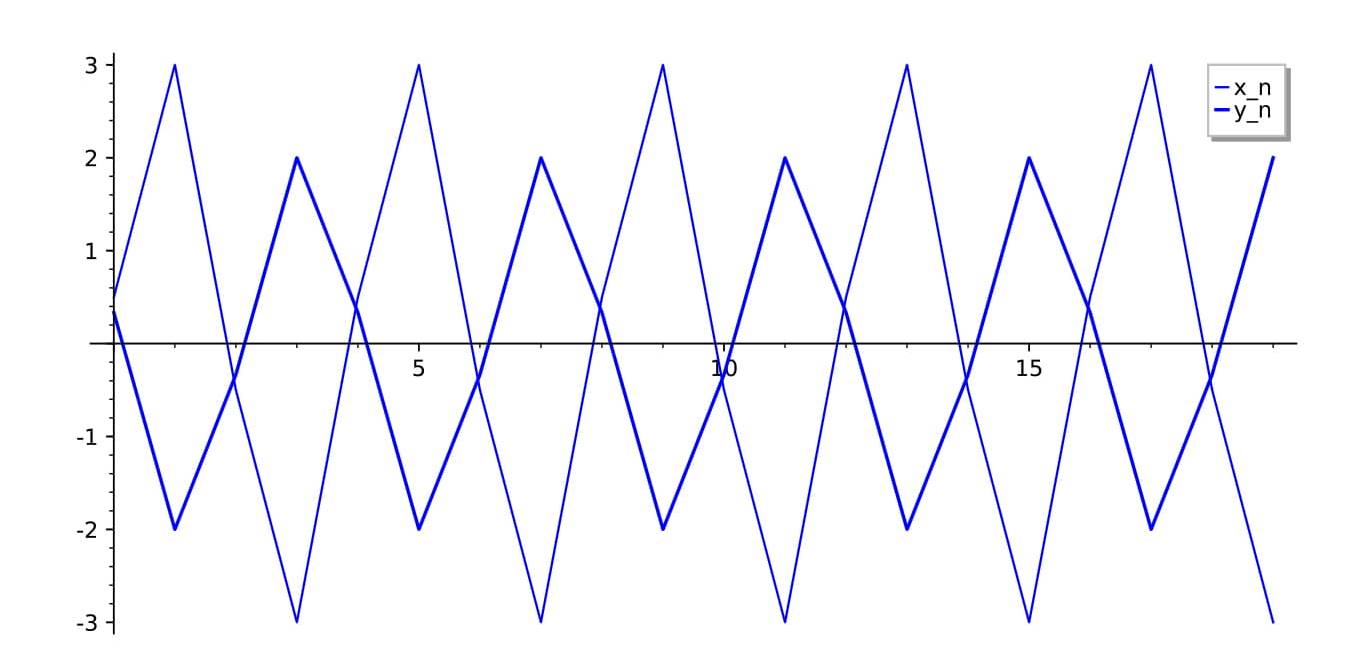}
        \caption{{$a=, b=, c=, d=, x_{-3}=, x_{-2}=, x_{-1}=,  x_0=, y_{-3}=,\,y_{-2}=, y_{-1}=, y_0=$.}}
        \label{graph44}
    \end{minipage}
\end{figure}
\begin{remark}
If $x_0 = - x_{-2}, y_0 = -y_{-2}, a = c = 1, d = -1$  then the solution of the system $x_{n+1}=\frac{x_{n-2}y_{n-1}}{y_{n}(a + bx_{n-2}y_{n-1})}, \quad y_{n+1}=\frac{y_{n-2}x_{n-1}}{x_{n}(c + dy_{n-2}x_{n-1})}$ is periodic with period four. The condition $x_{-1}y_{-2} = \frac{1 + a}{b}$ is not needed. This is clearly seen from the form of solution where one replaces $ab - b = 0$. This is the case of Theorem 18 of Elsayed \cite{EI}.
\end{remark}
\section{Conclusion}
We derived the Lie point symmetries of the difference equation \eqref{1.2}. The higher order equations were reduced to lower order equations, and via iterations, we were able to obtain solutions of the system of difference equations \eqref{1.1} in explicit form. The results found in this paper not only generalize solutions found by Elsayed and Ibrahim in \cite{EI}, but also greatly simplify the solutions in Theorems 5, 6, 7, 8, 9, 10, 11 and 12 in the same paper, by using only 8 equations instead of 16 equations.

\end{document}